\documentclass{article}

\usepackage[english]{babel}

\usepackage[letterpaper,top=2cm,bottom=2cm,left=3cm,right=3cm,marginparwidth=1.75cm]{geometry}

\usepackage{amsmath, amsfonts, amsthm,amssymb}
\usepackage{graphicx}
\usepackage[colorlinks=true, allcolors=blue]{hyperref}
\usepackage{mathtools} 

\usepackage{bbm}
\usepackage{comment}

\usepackage{thmtools}
\usepackage{thm-restate}

\usepackage{cancel}

\newtheorem{theorem}{Theorem}
\newtheorem{lemma}[theorem]{Lemma}
\newtheorem{proposition}[theorem]{Proposition}
\newtheorem{corollary}[theorem]{Corollary}
\theoremstyle{definition}
\newtheorem{definition}[theorem]{Definition}

\newtheorem{question}[theorem]{Question}
\newtheorem{example}[theorem]{Example}

\newcommand{\R}{\mathbb{R}}
\newcommand{\C}{\mathbb{C}}
\newcommand{\N}{\mathbb{N}}
\newcommand{\Z}{\mathbb{Z}}

\DeclareMathOperator{\val}{val}

\DeclareMathOperator{\trop}{trop}
\DeclareMathOperator{\row}{row}
\DeclareMathOperator{\im}{im}
\DeclareMathOperator{\vol}{vol}
\DeclareMathOperator{\pos}{pos}
\DeclareMathOperator{\conv}{conv}

\newcommand{\RP}{\mathbb{R}\lbrace\!\lbrace t\rbrace\!\rbrace}
\newcommand{\CP}{\mathbb{C}\lbrace\!\lbrace t\rbrace\!\rbrace}

\newcommand{\rev}[1]{\textcolor{black}{#1}}

\title{Tropical Toric Maximum Likelihood Estimation}
\author{Emma Boniface, Karel Devriendt and Serkan Ho\c{s}ten}

\begin{document}
\maketitle

\begin{abstract}
We consider toric maximum likelihood estimation over the
field of Puiseux series and study critical points of the likelihood function  using tropical methods. This problem translates to finding the intersection points of a tropical affine space with a classical linear subspace. We derive new structural results on tropical affine spaces and use these to give a complete and explicit description of the tropical critical points in certain cases. In these cases, we associate tropical critical points to the simplices in a regular triangulation of the polytope giving rise to the toric model.
\end{abstract}

\section{Introduction}\label{section: Introduction}
Algebraic statistics is concerned with the algebraic, geometric and combinatorial properties of statistical models. In the case of discrete distributions on a finite set $[n]=\lbrace 1,\dots,n\rbrace$, a statistical model corresponds to a subset $\mathcal{M}\subseteq\Delta_{n-1}$ of the probability simplex. If this model is parametrized by the image of a monomial map, it is called a log-linear or toric model \cite{Lectures_on_Alg_Stats, sullivant_2018_algebraic}.

Maximum likelihood estimation is a classical problem in statistics, which asks to find the distributions in the model $\mathcal{M}$ that best explain observed data. These distributions are called the maximum likelihood estimates (MLEs). In the discrete setting, the data is recorded in a vector $u\in\N^n$ and the MLEs are among the critical points $\hat{p}$ of the likelihood function
$$
\ell_u(p) = \frac{p_1^{u_1}\cdots p_n^{u_n}}{(p_1+\dots+p_n)^{\sum u_i}}$$ where $p$ ranges over $\overline{\mathcal{M}}_{reg} \setminus \mathcal{H}$, the regular locus of the projective closure $\overline{\mathcal{M}}$ of the model, minus the hypersurface $\mathcal{H}$ consisting of the poles and zeros of the likelihood function. The critical points of $\ell_u$ and their relation to the geometry of the model are a well-studied topic in algebraic statistics. In particular, the number of complex critical points is independent of the (generic) data $u$, and the resulting invariant is called the ML degree of the variety $\overline{\mathcal{M}}$ \cite{MLdegree, Solving_likelihood, huh_2014_likelihood}.
\\
~
\\
The tropical maximum likelihood estimation problem was introduced by Agostini et al. in \cite{agostini_2023_likelihood} and deals with finding the critical points of the likelihood function for data defined over the field $\RP$ of real Puiseux series. Elements in this field are formal power series $r=\sum_{i=1}^\infty c_it^{\alpha_i}$ with increasing rational exponents of bounded denominator. By thinking of $t$ as a small parameter, this setup can be used to model the asymptotic behavior of the maximum likelihood problem. The field of Puiseux series comes with the $t$-adic valuation $\val:r\mapsto\alpha_1$ that returns the smallest exponent in a given series. Starting from a data vector $u\in \RP^n$, the critical points of the likelihood function will lie in $(K^*)^n:=(\CP\backslash\lbrace 0\rbrace)^n$ and their asymptotic behavior is captured by their valuations. The goal of \emph{tropical maximum likelihood estimation} is to find the tropical critical points $\hat{q}:=\val(\hat{p})$ from the {tropical data vector} $w:=\val(u)$ directly. The adjective ``tropical'' reflects that this problem is naturally addressed using tools from tropical geometry \cite{maclagan_2015_introduction}. Throughout this paper we will assume that $w\geq0$ and there exists at least one index $i\in[n]$ for which $w_i = 0$. 

The tropical maximum likelihood estimation problem was solved for linear models by Agostini et al. \cite{agostini_2023_likelihood} and for a generalization to arbitrary matroids by Ardila, Eur and Penaguiao \cite{ardila_2023_tropical}. Both works find that the ML degree of the respective models is equal to Crapo's beta invariant of the matroid associated with the model. This reconfirms the ML degree computation for linear models found elsewhere; see \cite{MLdegree}. Furthermore, they find an explicit description of the multivalued map taking tropical data vector $w$ to tropical critical points $\hat{q}$ for some particular cases.
\\
~
\\
In this paper, we consider the tropical maximum likelihood estimation problem for toric models. A toric model is parametrized by an integer matrix 
$$ A = \begin{pmatrix} 
a_1 & a_2 & \cdots & a_n
\end{pmatrix}\in\Z^{k\times n}
$$
whose columns define a monomial map $\varphi_{A}$. We assume throughout the article that this matrix has full rank, \rev{that $\Z(A) = \Z^k$} and that its row span\rev{, denoted $\row(A)$,} contains the all ones vector. The monomial map gives rise to a projective toric variety $X_A = \overline{\im \varphi_A}\subseteq (K^*)^n \subset \mathbb{P}_K^{n-1}$ determined by the polytope $Q_A = \mathrm{conv}(a_i\mid i\in[n])$ and whose intersection with the probability simplex is a toric or log-linear statistical model. We recall \emph{Birch's theorem} which characterizes the MLE and the critical points of the likelihood function for toric models. 

\begin{theorem}[Birch's Theorem, {\cite[Ch. 7]{sullivant_2018_algebraic}}]  \label{theorem: Birch}
The MLE of the toric model $X_A \subset \C^n$ and data $u\in\R^n_{\geq 0}$ is the unique positive real point in the intersection of $X_A$ and $Y_{A,u}:= u + \ker(A)$. The critical points of $\ell_u$ on $X_A$ are the complex points in this intersection. When $u\in\RP^n$, the critical points of $\ell_u$ on $X_A \subset \CP^n$ are the intersection points with $Y_{A,u}$ in $\CP^n$.
\end{theorem}
We consider toric varieties $cX_A$ scaled by generic coefficients $c\in (K^*)^n$ with $\val(c)=0$. This means that the $i$th monomial in the parametrization of $X_A$ is scaled by $c_i$. In this case, it is known that the ML degree equals the degree $\deg(X_A)$ of the toric variety \cite{amendola_2019_maximum, huh_2014_likelihood}. Our focus will lie on describing the intersection points rather than counting them. We now give an example to illustrate the setup.
\begin{example}[Independent binary random variables]
\label{example: binary random variables - classical solution}
We consider the maximum likelihood problem for
$$
A = \begin{pmatrix}
1&1&1&1\\
0&1&0&1\\
0&0&1&1
\end{pmatrix} \quad\text{~and~}\quad u=(1,t^2,t,t^4)\in\RP^4.
$$
The toric variety $X_A$ is the quadric hypersurface $V(\langle p_1p_4-p_2p_3\rangle)\subseteq (\CP^*)^4$ of singular $2\times 2$ matrices. The corresponding log-linear model $X_A\cap \Delta_3$ describes all distributions on $[4]\cong\lbrace 0,1\rbrace^2$ of two independent binary random variables, which are characterized by a singular joint distribution matrix. The critical points of the likelihood function are found by intersecting $X_A$ with the affine line 
$$
Y_{A,u}=u+\mu\cdot\ker(A)=(1+\mu\,,\, t^2-\mu\,,\, t-\mu\,,\,t^4+\mu) \text{~for $\mu\in \CP$}.
$$
Introducing generic coefficients $c_1,\dots,c_4$, the intersection points are found by solving the quadratic equation $c_2c_3\cdot (1+\mu)(t^4+\mu) - c_1c_4\cdot (t^2-\mu)(t-\mu)=0$ in $\mu$, which results in the $2=\deg(X_A)$ solutions
\begin{align*}
\hat{p}_1 = (1+\alpha\,,\,-\alpha\,,\,-\alpha\,,\,\alpha)\,+\,\dots \quad\text{~and~}\quad
\hat{p}_2 = (1\,,\, t^2\,,\, t\,,\, \beta \cdot t^3)\,+\,\dots
\end{align*}
where higher order terms in $t$ are omitted, and with $\alpha=(c_1c_4/c_2c_3-1)^{-1}$ and $\beta=c_1c_4/c_2c_3$. Taking valuations of these solutions, we find the tropical critical points $\hat{q}_1=(0,0,0,0)$ and $\hat{q}_2=(0,2,1,3)$.
\end{example}

\rev{One motivation for passing to data vectors over the Puiseux series is to model sequences of observations. For instance, in fictitious repeated experiments in the setting of Example \ref{example: binary random variables - classical solution}, one might observe a data sequence $\left(\begin{smallmatrix}
6 & 1\\3 & 0\end{smallmatrix}\right), \left(\begin{smallmatrix}70 & 9\\20 & 1\end{smallmatrix}\right),\left(\begin{smallmatrix}919 & 16\\63 & 2 \end{smallmatrix}\right),\left(\begin{smallmatrix}9834 & 38\\123 & 5 \end{smallmatrix}\right), \dots$, counting outcomes of two binary variables. After normalization, this sequence can be modeled approximately as $u(t)=\big(\begin{smallmatrix}
1&t^2\\t&t^4
\end{smallmatrix}\big)\in\RP^4$ by thinking of $t$ as a small parameter. This is a data vector over the Puiseux series, whose valuation reflects the asymptotic rates of the observed data sequence. The classical critical points, computed for each data point in the observation sequence, will form a sequence which can similarly be modeled as  vectors $\hat{p}_i(t)$ over the Puiseux series. Their valuations, namely, the tropical critical points, in turn correspond to the asymptotic rates of the classical critical points along the sequence of experiments.} 

The tropical toric maximum likelihood estimation problem asks to circumvent computing the classical intersection \rev{as in Theorem \ref{theorem: Birch}} and instead find the tropical critical points $\val(cX_A\cap Y_{A,u})$ directly from the tropical data vector $w$. Because the generically scaled intersection and the valuation map commute in this setting \cite[Theorem 3.6.1]{maclagan_2015_introduction}, this problem translates to finding the intersection points between 
$$
\trop(cX_A):=\overline{\lbrace \val(x)\,:\, x\in cX_A\rbrace}\quad\text{~and~}\quad\trop(Y_{A,u}):=\overline{\lbrace\val(y)\,:\,y\in Y_{A,u}\rbrace}
$$ 
where the Euclidean closure is taken inside $\mathbb{R}^n$. By the fundamental theorem of tropical geometry \cite[Theorem 3.2.5]{maclagan_2015_introduction}, these pointwise valuations are tropical varieties. The tropical toric variety $\trop(cX_A)=\row(A)$ is a classical linear subspace while the tropical affine space $L_{A,u}:=\trop(Y_{A,u})$ is a polyhedral complex whose combinatorial structure is governed by a matroid. The algebraic problem of intersecting a toric variety with an affine subspace is thus replaced by the combinatorial problem of intersecting a linear subspace with a polyhedral complex. Throughout the article, $e_i$ will denote the $i$th basis vector of $\R^n$, $e_\tau=\sum_{i\in\tau}e_i$ the indicator vector of a subset $\tau\subseteq[n]$, and $\pos(e_i \mid i\in\tau)= \{\sum_{i\in\tau}\lambda_ie_i \, : \,  \lambda_i\geq 0\}$ the cone with rays $(e_i)_{i\in\tau}$. We revisit Example \ref{example: binary random variables - classical solution} from the tropical perspective.

\begin{example}[Independent binary random variables revisited]\label{example: binary random variables - tropical solution}
We consider the tropical maximum likelihood estimation problem for the independent binary random variables model, with $A$ as defined in Example \ref{example: binary random variables - classical solution} and tropical data vector $w=(0,2,1,4)$ equal to $\val(1,t^2, t, t^4)$. The tropical toric variety is the linear subspace $\row(A)$ and the tropical affine space $L_{A,u}$ is a one-dimensional polyhedral complex in $\R^4$ with three vertices
$$
v_0=(0, \, 1, \, 1, \,1),\qquad v_1=(0,\, 0, \, 0, \, 0),\qquad v_2=(0,\, 2,\, 1,\, 2),
$$
two bounded line segments $\conv(v_0,v_1)$ and $\conv(v_0,v_2)$ and five cones
$$
v_0+\pos(e_3),  \qquad v_1+\pos(e_1),  \qquad v_2+\pos(e_2),  \qquad v_2+\pos(e_4),  \qquad 
v_1 - \pos(e_{[4]}).
$$
We computed this complex using \texttt{polymake} \cite{gawrilow_2000_polymake} based on Proposition \ref{proposition: homogenization}, which expresses $L_{A,u}$ as the coordinate subspace of a tropical linear space.
The polyhedral complex $L_{A,u}$ intersects $\row(A)$ in two points
$$
\hat{q}_1 = (0, \, 0, \, 0,\, 0),\qquad\hat{q}_2 = (0,\, 2,\, 1,\, 3),
$$
which lie on the vertex $v_1$ and the cone $v_2+\pos(e_4)$. These points correspond to the valuation of the critical points $\hat{p}_1$ and $\hat{p}_2$ computed in Example \ref{example: binary random variables - classical solution} by intersecting $cX_A$ and $Y_{A,u}$.  
\end{example}

Tropical affine spaces were studied by Braden et al. \cite{braden_2022_semismall} as ``augmented Bergman fans" and their connection to tropical geometry was made explicit in \cite[\S 5.3]{eur_2023_stellahedral}. The tropical affine spaces considered in this article are more general than augmented Bergman fans: the latter arise as $\val(Y_{A,u})$ with both $A$ and $u$ defined over fields with trivial valuation, while in our case the data lives in $u\in\RP^n$ with nontrivial valuation in general. This distinction is similar to the difference between Bergman fans and general tropical linear spaces arising from valuated matroids; see for instance \cite[\S 4.4]{maclagan_2015_introduction}. \rev{Analogously to the linear setting, augmented Bergman fans appear as recession fans of general tropical affine spaces \cite[Theorem 3.5.6 and Theorem 4.4.4]{maclagan_2015_introduction}.}
\\
~
\\
The tropical affine space $L_{A,u}$ is a pure $(n-k)$-dimensional polyhedral complex in $\R^n$ whose structure is determined by the matroid $M(A)$ associated with matrix $A$. This matroid records all $k$-subsets of $[n]$ for which the corresponding columns in $A$ are linearly independent; these are called bases and are denoted $\tau\in M(A)$. As our first contribution, we give a new structural result on tropical affine spaces by explicitly identifying \rev{a subset of the} cones in $L_{A,u}$ in terms of the data $A$ and $u$. 
\begin{restatable}{theorem}{conesLAu}\label{theorem: cones in LAu}
The tropical affine space $L_{A,u}$ contains the \rev{cones} $C_\tau:=w^{(\tau)}+\pos(e_i\mid i\in[n]\backslash\tau)$ for every basis $\tau\in M(A)$. Here, $w^{(\tau)}$ is a vertex of $L_{A,u}$ with coordinates
$$
\begin{cases}
w^{(\tau)}_j := \min\lbrace w_j\,,\, w_i\,:\, i\in [n]\setminus\tau\text{~s.t.~}\tau-j+i\in M(A) \rbrace &\textup{~for all $j\in\tau$}\\
w^{(\tau)}_i := \max\lbrace w_j^{(\tau)}  \,:\, j\in\tau\text{~s.t.~}\tau-j+i\in M(A)\rbrace &\textup{~for all $i\in[n]\backslash\tau$}
\end{cases}
$$
\end{restatable}
The main novelty in this theorem is the expression for the vertices $w^{(\tau)}$. We arrive at this result by homogenizing $L_{A,u}$ to a tropical linear space and then using certain properties that this space inherits from the homogenization procedure. We also study the endomorphism of $\R^n$ which takes a tropical data vector $w$ to a vertex $w^{(\tau)}$, and show that this map is continuous, piecewise linear, shift invariant, idempotent and contractive. 

Theorem \ref{theorem: cones in LAu} is relevant for the tropical maximum likelihood estimation problem because the intersection of $C_{\tau}$ and $\row(A)$ can be controlled by considering certain subdivisions of the polytope $Q_A \rev{ = \mathrm{conv}(a_i\mid i\in[n])}$. This leads to explicit expressions for some or all tropical critical points in terms of $w$. A polyhedral subdivision $\Delta$ of $Q_A$ is a collection of subsets of $[n]$, called cells, such that the polytopes $(\conv(a_i\mid i\in\sigma))_{\sigma\in\Delta}$ cover $Q_A$ with certain additional criteria; see Section \ref{section: subdivisions and tropical critical points} for further details. A regular subdivision $\Delta_\omega$ of $Q_A$ induced by the weight vector $(\omega_{i})_{i\in[n]}$ is the polyhedral subdivision whose cells correspond to the lower faces in the convex hull of  $\{(a_i,\omega_i)\}_{i\in[n]}\subseteq\mathbb{R}^{k+1}$. We say that $\tau$ lies in a cell if $\tau\subseteq\sigma$ for some cell $\sigma$. The normalized volume of $\tau$, $\vol_A(\tau)$, is equal to $|\det(A_\tau)|$.
\begin{restatable}{proposition}{ConeIntersection}\label{proposition:subdivision condition for cone intersection}
The linear subspace $\row(A)$ intersects $C_{\tau}$ if and only if $\tau$ lies in a cell of $\Delta_{-w^{(\tau)}}$. The intersection point, if it exists, is $\hat{q}(\tau):=A^T(A_{\tau}^T)^{-1}w^{(\tau)}_\tau$ and has multiplicity $\vol_A(\tau)$.
\end{restatable}
This proposition reflects the interaction between the geometry of the polytope $Q_A$ and the tropical data $w$ in the tropical toric maximum likelihood estimation problem. It leads to the following sufficient criterion that determines whether all tropical critical points lie on the cones $C_\tau$. A triangulation is a subdivision whose cells are simplices.
\begin{restatable}{theorem}{ConditionForAllIntersections}\label{theorem: tropical critical points from triangulation}
If $Q_A$ admits a regular triangulation $\Delta$ such that every maximal cell $\tau\in\Delta$ lies in a cell of $\Delta_{-w^{(\tau)}}$, then the tropical critical points are $\hat{q}(\tau)$, each with multiplicity $\vol_A(\tau)$, where $\tau$ runs over all maximal cells in $\Delta$.
\end{restatable}
In general such a triangulation may not exist, but when it does, Theorem \ref{theorem: tropical critical points from triangulation} provides a full description of all tropical critical points in terms of the tropical data. We use this theorem to give complete and explicit solutions in two specific cases. Both cases again display interactions between the geometry of $Q_A$, the combinatorics of $M(A)$ and the choice of tropical data $w$. Let $\mathcal{O}(w):=\lbrace i\in[n]\,:\, w_i=0\rbrace$ be the indices where the tropical data is minimal. As a first result, we show that the tropical critical points are trivial when this minimal set contains a basis.
\begin{restatable}{theorem}{OContainsBasis}\label{theorem: O contains basis zero solutions}
If $\mathcal{O}(w)$ contains a basis of $M(A)$, then $\hat{q}=0$ is the unique tropical critical point \rev{and it has} multiplicity $\vol_A(Q_A)$.
\end{restatable}
For the second result, consider a matrix $A$ whose columns are  in general position. Here, all $k$-subsets of columns form a basis and $M(A)$ is a uniform matroid. In this setting, we find a complete solution for the tropical maximum likelihood problem if some further restrictions on the ``spread'' of the $k$ smallest entries of the tropical data vector $w$ are satisfied.
\begin{restatable}{theorem}{SolutionForUniformMatroids}\label{theorem: solution for uniform matroids}
For $M(A)$ uniform of rank $k$, the following holds:
\begin{enumerate}
    \item If $\mathcal{O}(w)$ is not a face of $Q_A$, then the \rev{unique} tropical critical point is $\hat{q}=0$ with multiplicity $\vol_A(Q_A)$.
    
    \item If $\mathcal{O}(w)$ is a face of $Q_A$, then there exists a constant $c_{A,\mathcal{O}}\geq 1$ such that for all tropical data vectors that have $k$ entries with $w_j\leq c_{A,\mathcal{O}}\cdot\min\{w_i:w_i>0\}$, the tropical critical points are $\hat{q}(\tau)$, each with multiplicity $\vol_A(\tau)$, where $\tau$ runs over all maximal cells in an \rev{arbitrary} regular triangulation that refines $\Delta_{e_{\mathcal{O}(w)}}$.
\end{enumerate}
\end{restatable}

\begin{example}[Independent binary random variables revisited]\label{example: illustration for uniform matroids} The matrix $A$ in Example \ref{example: binary random variables - classical solution} corresponds to a configuration of four points in the plane whose convex hull $Q_A$ is the unit square.
$$
A \,=\, \begin{pmatrix}
1&1&1&1\\0&1&0&1\\0&0&1&1
\end{pmatrix}\quad\quad Q_A \,\,=\,\, \raisebox{-0.35\height}{\includegraphics[width=0.08\textwidth]{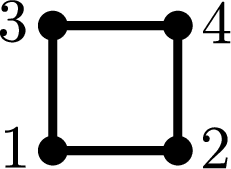}}
$$
The matroid $M(A)$ is uniform and for tropical data $w=(0,2,1,4)$ we find that $\mathcal{O}(w)=\lbrace 1\rbrace$ is a face of $Q_A$. Using details from the proof of Theorem \ref{theorem: solution for uniform matroids} in Section \ref{section: subdivisions and tropical critical points}, we compute the constant $c_{A,\mathcal{O}}=2$. Since $w$ has $k=3$ entries smaller than or equal to $c_{A,\mathcal{O}}\cdot\min\{w_j>0\}=2\cdot 1$, namely $(w_1,w_2,w_3)=(0,2,1)$, Theorem \ref{theorem: solution for uniform matroids} applies. The computation of the tropical critical points based on the simplices $\tau_1=123$ and $\tau_2=234$ in the triangulation $\Delta_{e_1}$ is shown below. 
$$
\raisebox{-0.35\height}{\includegraphics[width=0.085\textwidth]{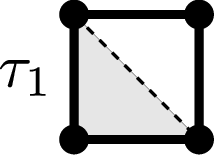}} 
 \quad \begin{array}{c}
w^{(\tau_1)}=(0,2,1,2)\\
\hat{q}(\tau_1)=(0,2,1,3)
\end{array}\qquad\qquad\qquad
\raisebox{-0.35\height}{\includegraphics[width=0.085\textwidth]{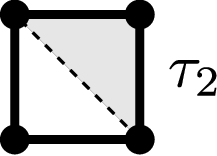}} 
 \quad \!\!\begin{array}{c}
w^{(\tau_2)}=(0,0,0,0)\\
\hat{q}(\tau_2)=(0,0,0,0)
\end{array}
$$
\end{example}
~\\
\noindent \textbf{Organization.}
The rest of this article is organized as follows: Section \ref{section: tropical toric MLE} formally introduces the tropical toric maximum likelihood problem and the necessary background on tropical geometry. In Section \ref{section: tropical affine spaces}, we study in more detail the tropical affine space $L_{A,u}$ and prove our main structural result, Theorem \ref{theorem: cones in LAu}. We then use this result in Section \ref{section: subdivisions and tropical critical points} to show how certain subdivisions of the polytope associated with the toric model can be used to find the tropical critical points. This leads to a complete solution of the tropical toric maximum likelihood problem in specific cases. Section \ref{section: further examples} contains a number of examples to illustrate and apply the main results. In Section \ref{section: tIPS} we consider a tropicalization of the classical ``iterative proportional scaling'' algorithm (IPS) to compute toric MLEs. We introduce the tropical IPS and discuss some experimental observations. \rev{Finally, in Section \ref{section: general position} we show that when the columns of $A$ are in general position, Theorem \ref{theorem: tropical critical points from triangulation} applies for a full-dimensional set of tropical data vectors $w$. }
\\~\\
\textbf{Acknowledgements} 
We thank Bernd Sturmfels for proposing the problem that lead to this paper.

\section{The tropical toric maximum likelihood estimation problem}\label{section: tropical toric MLE}
In this section we formally introduce the tropical toric maximum likelihood estimation problem. This problem is given by two pieces of data:
\begin{itemize}
    \item[(1)] An integer matrix $A\in\mathbb{Z}^{k\times n}$ of full rank $k$ and with $(1,\dots,1)\in\row(A)$.
    \item[(2)] A data vector $u\in \RP^n$ which we assume to be generic and to satisfy $\val(\sum_{i=1}^n u_i) = 0$.
\end{itemize}
Classically, the maximum likelihood estimation problem is based on data defined over $\R$ and asks to identify the complex critical points of the likelihood function $\ell_u$. In our setting, the data is defined over the field of real Puiseux series $\RP$ and the critical points will lie in its algebraic closure $K:=\CP$. Elements of $\RP$ are formal power series $\sum_{i=1}^{\infty}c_it^{\alpha_i}$ with coefficients in $\R$ and rational exponents $\alpha_1<\alpha_2<\dots$ that have a bounded common denominator. This field comes with a valuation $\val:\RP\rightarrow \R\cup\{\infty\}$, defined as
$$
\val\big(\,\sum_{i=1}^\infty c_it^{\alpha_i}\big) \,=\, \min\lbrace \infty\,,\,\alpha_i : i\geq 0\text{~s.t.~}c_i\neq 0\rbrace.
$$ 
See \cite[Chapter 2.1]{maclagan_2015_introduction} for more details. We recall from Birch's theorem that the critical points for a toric model are found by intersecting the underlying toric variety with an affine subspace. This still applies when working over $\CP$. More precisely, our varieties live in the algebraic torus $(K^*)^n$ with $K^*=K\setminus\{0\}$, and are defined by ideals in the ring of Laurent polynomials $K[x_1^{\pm 1},\dots,x_n^{\pm 1}]$. We use the standard abbreviation $x^{v}=x_1^{v_1}\cdots x_{n}^{v_n}$ for a monomial with exponent vector $v\in\Z^n$.

\begin{definition}
The scaled toric variety $cX_A=V(I_{A,c})\subseteq (K^*)^n$ is defined by the binomial ideal
$$
I_{A,c}:=\left\langle c^b\cdot x^a-c^a\cdot x^b\,:\, a,b\in\N^n\text{~s.t.~}A(a-b)=0 \right\rangle\subseteq K[x_1^{\pm 1},\,\dots,\,x_n^{\pm 1}].
$$
Here, we assume that the coefficients $c\in(K^*)^n$ are generic and have $\val(c)=0$. The affine subspace $Y_{A,u}=V(J_{A,u})\subseteq(K^*)^n$ is defined by the inhomogeneous ideal
$$
J_{A,u} := \langle f-f(u)\,:\,f\in\row(A)\rangle \subseteq K[x_1^{\pm 1},\,\dots,\,x_n^{\pm 1}].
$$
\end{definition}
By Birch's theorem, the critical points of the likelihood function on $cX_A$ are the intersection points of $cX_A$ with the affine subspace $Y_{A,u}$. These are points in $(K^*)^n$ and thus have valuations in $\mathbb{R}^n$.
\begin{definition}
The \textit{tropical critical points} of $cX_A$ and data $u \in \RP^n$ are the vectors $\hat{q}=\val(\hat{p})\in\mathbb{R}^n$, where $\hat{p}$ runs over all intersection points of $cX_A$ and $Y_{A,u}$ in $(K^*)^n$. The number of times a given tropical critical point $\hat{q}$ appears as the valuation of an intersection point is called its \textit{multiplicity}.
\end{definition}
The main goal of \emph{tropical toric maximum likelihood estimation} is to find the tropical critical points \rev{and their multiplicities} directly from the valuation \rev{$w=\val(u)$ of the data vector; we call this third piece of data the tropical data vector:
\begin{itemize}
    \item[(3)] A tropical data vector $w\in\R^n$ is the valuation of a data vector $u$ and satisfies $\min_{i\in [n]}w_i = 0.$
\end{itemize}}
We now introduce the relevant background on tropical geometry; see \cite{maclagan_2015_introduction} for details. For a Laurent polynomial $f=\sum_{a\in\mathbb{Z}^n}c_ax^a$ in $K[x_1^{\pm 1},\dots,x_n^{\pm 1}]$, the tropical hypersurface $\trop(V(f))$ is defined as
$$
\trop(V(f)) := \Big\lbrace x \in \R^n : \text{~the minimum is achieved twice in~}\min_{a\in\mathbb{Z}^n} (\langle a,x\rangle + \val(c_a))\Big\rbrace
$$
with $\langle\cdot,\cdot\rangle$ the standard inner product on $\R^n$. These are the building blocks of tropical varieties.
\begin{definition}
Let $X$ be a variety in the algebraic torus defined by the ideal $I\subseteq K[x_1^{\pm 1},\dots,x_n^{\pm 1}]$. The \textit{tropicalization} of $X$ is defined as
$ \trop(X) := \bigcap_{f\in I}\trop(V(f))\subseteq\R^n$. A tropical variety is any subset of $\R^n$ that arises as the tropicalization of a subvariety of $(K^*)^n$.
\end{definition}
\emph{The fundamental theorem of tropical algebraic geometry} bridges classical algebraic varieties and their tropical counterparts through valuations. 
\begin{theorem}[{\cite[Theorem 3.2.5]{maclagan_2015_introduction}}]
Let $I$ be an ideal in $K[x_1^{\pm 1},\dots,x_n^{\pm 1}]$ and $X=V(I) \subset (K^*)^n$ be the variety defined by $I$. Then the following subsets of $\R^n$ coincide:
\begin{enumerate}
    \item The tropical variety $\trop(X)$ and
    \item The Euclidean closure $ \overline{\lbrace\val(x)\,:\, x\in X\rbrace}$ of pointwise valuations of $X$.
\end{enumerate}
\end{theorem}
\rev{In the next corollary the intersection $\cap_{st}$ refers to the stable intersection of two tropical varieties as balanced polyhedral complexes \cite[Definition 3.6.5]{maclagan_2015_introduction}.}
\begin{corollary} \label{cor: intersection commutes with trop}
The set of tropical critical points of $cX_A$ for the data $u$ is equal to 
$$ \trop \left( cX_A \cap Y_{A,u}\right) \, = \, \trop(cX_A) \cap_{st} \trop(Y_{A,u}) \, = \, \row(A) \cap_{st} L_{A,u}.$$
Moreover, for generic data $u$, the number of the tropical critical points counted with multiplicity is equal to $\deg(cX_A) = \vol_A(Q_A)$.
\end{corollary}
\begin{proof}
The first equality follows from the stable intersection theorem for tropical varieties \cite[Theorem 3.6.1]{maclagan_2015_introduction}. This theorem requires $c$ to be generic with $\val(c) = 0$ as we assume in our setting. From the monomial parametrization of $cX_A$ it follows that $\trop(cX_A)$ is equal to 
$\row(A)$, the row span of $A$, and $L_{A,u}=\trop(Y_{A,u})$ by definition. This gives the second equality. 
The last statement is a consequence of \cite[Theorem 13]{amendola_2019_maximum} and the genericity of $c$ and $u$.
\end{proof}
\begin{proposition} \label{prop: computing multiplicity}
If a tropical critical point in $\trop(cX_A \cap Y_{A,u})$ is obtained from
the intersection $\row(A) \cap_{st} \Gamma$ where $\Gamma$ is a maximal face of $L_{A,u}$, the contribution to its multiplicity is equal to the index of the lattice generated by the rows of $A$ and the lattice of integer vectors in the linear subspace parallel to the affine hull of $\Gamma$. The multiplicity of a tropical critical point is the sum of these contributions.
\end{proposition}
\begin{proof}
Following \cite[Proposition 3.6.12]{maclagan_2015_introduction}, the stable intersection $\row(A) \cap_{st} L_{A,u}$ can be computed  via 
$$\lim_{\varepsilon \rightarrow 0} (\varepsilon v + \row(A)) \cap L_{A,u}$$
where $v \in \R^n$ is general. 
If the maximal face $\Gamma$ of $L_{A,u}$ is involved in this intersection, the multiplicity of the corresponding stable intersection point is computed as follows. According to Definition 3.4.3 in \cite{maclagan_2015_introduction} one first needs to determine the multiplicities of the associated primes of the initial ideal $\mathrm{in}_w(J_{A,u})$, for any $w$ in the relative interior of $\Gamma$. Since $J_{A,u}$ is a linear ideal, so is the initial ideal $\mathrm{in}_w(J_{A,u})$ and thus it is prime. This means that its multiplicity is equal to one. Similarly, $\mathrm{in}_w(I_{A,c}) = I_{A,c}$
for any $w \in \row(A)$, and because $I_{A,c}$ is also a prime ideal this multiplicity is also one. Finally, it follows from \cite[Definition 3.6.5]{maclagan_2015_introduction} that the multiplicity of $\row(A) \cap \Gamma$ is the index of the lattice generated by the rows of $A$ plus the lattice of integer vectors in the linear subspace parallel to the affine hull of $\Gamma$. 
\end{proof}
Computing the multiplicities of tropical critical points comes down to computing a certain determinant, as shown in the following example.
\begin{example}[Independent binary random variables continued]
We continue with Example \ref{example: binary random variables - tropical solution}. 
For the choice of $v=(1,2,-1,3)$ and $0 < \varepsilon \ll 1$, 
the stable intersection $\row(A) \cap_{st} L_{A,u}$  
yields two points:
$$ (3\varepsilon, 0,0,0) = \varepsilon(1,2,-1,3) + 2\varepsilon(1,1,1,1) - 4\varepsilon(0,1,0,1) - \varepsilon(0,0,1,1) = (0,0,0,0) + 3\varepsilon(1,0,0,0)$$
$$ (0,2,1,3+3\varepsilon) = \varepsilon(1,2,-1,3) - \varepsilon (1,1,1,1) + (2 - \varepsilon)(0,1,0,1) + (1+2\varepsilon)(0,0,1,1) = (0,2,1,2) + (1+3\varepsilon)(0,0,0,1)$$ 
The first point lies on the cone 
$(0,0,0,0) + \lambda (1,0,0,0)$ 
and the second on the cone 
$(0,2,1,2) + \lambda(0,0,0,1)$
belonging to $L_{A,u}$. As $\varepsilon \rightarrow 0$, we recover the tropical critical points $\hat{q}_1 = (0,0,0,0)$
and $\hat{q}_2 = (0,2,1,3)$.
We compute the respective multiplicities by
$$ \operatorname{mult}(\hat{q}_1) = \det \begin{pmatrix}
    1 & 1 & 1 & 1 \\
    0 & 1 & 0 & 1 \\
    0 & 0 & 1 & 1 \\
    1 & 0 & 0 & 0
\end{pmatrix} = 1, 
\qquad\qquad
\operatorname{mult}(\hat{q}_2) = \det \begin{pmatrix}
    1 & 1 & 1 & 1 \\
    0 & 1 & 0 & 1 \\
    0 & 0 & 1 & 1 \\
    0 & 0 & 0 & 1
\end{pmatrix} = 1.
$$
We note that to compute the stable intersection, one needs a generic, i.e., random  vector $v$. In practice one would choose a few random vectors and try each of them. If they all give the same result one can be almost sure that any one of them works. We chose this strategy in this example.
\end{example}
Corollary \ref{cor: intersection commutes with trop} forces us to study the tropical 
affine space $L_{A,u}$ which we will take on in the next section.   

\section{Tropical affine spaces}
\label{section: tropical affine spaces}
We will now construct the tropical affine space $L_{A,u}=\trop(Y_{A,u})$ more carefully, leading up to the proof of our main structural result, Theorem \ref{theorem: cones in LAu}. To simplify our analysis, we will express $L_{A,u}$ in terms of its homogenization $L_{A,u}^h$, which is a tropical linear space. 
\begin{proposition}\label{proposition: homogenization}
The tropical affine space $L_{A,u}$ is equal to the \rev{ intersection $L_{A,u}^h 
\cap \{x_{n+1} = 0\}$ where $L_{A,u}^h =\trop(\row(B^h))\subseteq \R^{n+1}$}. Here, $B^h=\left(\begin{smallmatrix*}[r]
B&0\\u^T&-1
\end{smallmatrix*}\right)$ is a $(n-k+1)\times(n+1)$ matrix and $B$ is a matrix whose rows span $\ker(A)$.
\end{proposition}
\begin{proof}
We recall the definition of $L_{A,u}$ as a tropical variety
$$
L_{A,u} = \bigcap_{f\in J_{A,u}} \big\lbrace x\in\R^n \,:\, \text{min. is achieved twice in}\, \min(\val(d)\,,\,x_i+\val(c_i): i\in [n])\big\rbrace\subseteq \R^n.
$$
Here, we write the inhomogeneous polynomials $f\in J_{A,u}$ as $f=\sum_{i=1}^n c_i x_i + d$. The homogenization of $J_{A,u}$ is a linear ideal $J_{A,u}^h\subseteq K[x_1^{\pm 1},\dots,x_n^{\pm 1},\rev{x_{n+1}^{\pm 1}}]$ whose elements can be written as $g=\sum_{i=1}^nc_ix_i+d \rev{x_{n+1}}$. This ideal determines the tropical linear space $L_{A,u}^h=\trop(Y_{A,u}^h)$, equal to
$$
L_{A,u}^h = \bigcap_{g\in J_{A,u}^h}\big\lbrace (x,\rev{x_{n+1}})\in\R^{n+1} \,:\, \text{min. is achieved twice in}\, \min(\rev{x_{n+1}}+\val(d)\,,\,x_i + \val(c_i) : i\in [n])\big\rbrace \subseteq\R^{n+1}.
$$
Clearly a point $x$ lies in $L_{A,u}$ if and only if $(x,0)$ lies in $L_{A,u}^h$. We now show that $Y_{A,u}^h=\row(B^h)$. The affine space $Y_{A,u}=u+\ker(A)$ is homogenized to the linear space $Y_{A,u}^h=\ker(A^h)\subseteq (K^*)^{n+1}$, where $A^h=(A\mid -Au)$. The proposed matrix $B^h$ satisfies $(A^h)^TB^h=0$, and thus $\row(B^h)=\ker(A^h)$.
\end{proof}
Tropical linear spaces are well studied and their structure depends on the matroid associated with the linear space.
\begin{definition}
A \textit{matroid} $M$ on $[n]$ is a nonempty collection of subsets of $[n]$ which are called the \textit{bases} of $M$ such that the following \emph{basis-exchange axiom} is satisfied: for any two bases $\tau_1,\tau_2$ of $M$ and all $i_1\in\tau_1\setminus \tau_2$, there exists $i_2\in\tau_2\setminus\tau_1$ such that $\tau_1-i_1+i_2$ is a basis of $M$.    
\end{definition} 
It follows from the definition that all bases of $M$ have the same size, which is called the \textit{rank} of the matroid. We will write $\tau\in M$ to denote that the subset $\tau\subseteq [n]$ is a basis of $M$ and make use of the following stronger version of the basis-exchange axiom \cite{brualdi_1969_comments}.
\begin{proposition}[symmetric basis exchange]
For any two bases $\tau_1,\tau_2$ of a matroid $M$ and all $i_1\in\tau_1\setminus \tau_2$, there exists $i_2\in\tau_2\setminus\tau_1$ such that $\tau_1-i_1+i_2$ and $\tau_2-i_2+i_1$ are bases of $M$.
\end{proposition}
The matroid $M(A)$ associated with a nonsingular matrix $A\in\Z^{k\times n}$ is the matroid on $[n]$ of rank $k$ whose bases are given by those $k$-subsets for which the corresponding columns are linearly independent, i.e., $\tau\in M(A) \iff \det(A_\tau)\neq 0$. The following three matroids will be relevant to us: $M(A)$ is the rank $k$ matroid on $[n]$ determined by the matrix $A$ that specifies the toric model; bases of $M(A)$ will be denoted by $\tau$. The matroid $M(B)$ on $[n]$ has rank $(n-k)$ and is dual to $M(A)$, which means that its bases are given by $\sigma = [n]\setminus \tau$ for all $\tau\in M(A)$. By our assumption that $\row(A)$ contains the all ones vector, matrix $A$ has no zero columns. As a result, every element of $[n]$ is contained in at least one basis of $M(A)$ and no element of $[n]$ is contained in every basis of $M(B)$; we say that $M(A)$ has no \textit{loops} and $M(B)$ has no \textit{coloops}. Finally $M(B^h)$ is the matroid on $[n+1]$ of rank $(n-k+1)$ whose bases will be denoted by $\gamma$.
\begin{proposition}\label{proposition: matroid M(B')}
For generic $u$, the bases of $M(B^h)$ are $\lbrace \sigma+j\,:\, \sigma\in M(B),\,j\in [n+1]\setminus \sigma\rbrace$.
\end{proposition}
\begin{proof}
By genericity of $u$, the rank of $B^h$ is $n-k+1$ and the bases of the matroid $M(B^h)$ are given by the $(n-k+1)$-subsets $\gamma\subseteq [n+1]$ of columns for which the determinant $\det(B^h_\gamma)$ is nonzero. 
If $\gamma$ contains $(n+1)$ then $\det(B^h_\gamma) = \pm\det(B_{\gamma-(n+1)})$. Hence $\gamma$ is a basis of $M(B^h)$ if and only if $\gamma-(n+1)$ is a basis of $M(B)$. Otherwise we write the determinant using Laplace expansion along the last row as
\begin{equation}\label{eq: Laplace expansion of determinant}
\det(B^h_{\gamma}) = \sum_{j\in\gamma} (-1)^{t(j)} \det(B_{\gamma-j})\cdot u_j
\end{equation}
for some bijection $t:\gamma\mapsto[n-k+1]$. By genericity of $u$, this determinant is nonzero if and only if $\det(B_{\gamma-j})\neq 0$ for some $j$ or, equivalently, if there exists $j\in\gamma$ such that $\gamma-j=\sigma\in M(B)$.
\end{proof}
The matroid $M(B^h)$ is called the free coextension of $M(B)$; see \cite{eur_2023_stellahedral}. We note that the element $(n+1)$ plays a special role for this matroid. This is related to the role of the $(n+1)$st coordinate of $\R^{n+1}$ as a homogenization variable. The final ingredient to define tropical linear spaces are tropical Pl\"{u}cker coordinates.
\begin{definition}
Let $\row(\rev{\Lambda})\subseteq (K^*)^n$ be a linear subspace determined by matrix $\rev{\Lambda}$. The \textit{tropical Pl\"{u}cker vector} is the function $\pi:\lbrace\text{bases of $M(\rev{\Lambda})$}\rbrace\rightarrow \R$ determined by $\pi_\tau = \val(\det(\rev{\Lambda}_\tau))$.
\end{definition}
\begin{proposition}
\label{proposition: tropical plucker vector}
\rev{For generic $u$}, the tropical Pl\"{u}cker vector of $\row(B^h)\subseteq (K^*)^n$ is equal to
$$
\pi_{\gamma} = \min\{ w_i\,:\, i\in\gamma\text{~s.t.~}\gamma-i\in M(B) \},
$$
for all bases $\gamma$ that do not contain $(n+1)$, and $\pi_{\gamma}=0$ for all bases $\gamma$ that contain $(n+1)$.
\end{proposition}
\begin{proof}
We recall expression \eqref{eq: Laplace expansion of determinant} for the determinant $\det(B^h_\gamma)$ when $\gamma$ does not contain $(n+1)$
$$
\det(B^h_\gamma) = \sum_{i\in\gamma}(-1)^{t(i)}\det(B_{\gamma-i})\cdot u_i = \sum_{\substack{i\in\gamma\,:\,\gamma-i\in M(B)}}(-1)^{t(i)}\det(B_{\gamma-i})\cdot u_i.
$$
\rev{By genericity of $u$,} the valuation of this sum gives the proposed minimization expression for $\pi_{\gamma}$.  For a basis $\gamma=\sigma+(n+1)$ with $\sigma\in M(B)$, we have $\det(B^h_\gamma) =  \pm\det(B_\sigma)$ whose valuation is zero and thus $\pi_{\gamma}=0$.  \rev{We note that for any fixed valuation $w=\val(u)$, the required genericity condition on the data vector $u$ is that the expression above for the determinant $\det(B^h_{\gamma})$ has no unexpected cancellations. For instance, if $u$ satisfies $(-1)^{t(i)}\det(B_{\gamma-i})u_i+(-1)^{t(i')}\det(B_{\gamma-i'})u_{i'}=0$ for some $i,i'\in\gamma$, this would lead to $w_i,w_{i'}$ not appearing in the expression for $\pi_{\gamma}$. Such cancellations occur when $u$ lies on a union of hypersurfaces cut out by equations as above, which is avoided by a generic data vector.}
\end{proof}
We will use the following characterization of tropical linear spaces.
\begin{lemma}[{\cite[Lemma 4.4.7]{maclagan_2015_introduction}}]\label{lemma: tropical linear space initial matroid}
Let $L$ be a tropical linear space associated with matroid $M$ on $[n]$ and $\pi$ the tropical Pl\"{u}cker vector. Then a point $x\in\mathbb{R}^n$ lies in $L$ if and only if the collection of bases
$$
\Big\lbrace \tau \in M \,:\, \langle e_{\tau},x\rangle - \pi_\tau \textup{~is maximal \rev{over all bases of} M}\Big\rbrace
$$
covers every element of $[n]$. 
\end{lemma}
We are now ready to prove our main structural result on the tropical affine space $L_{A,u}$.

\conesLAu*
\begin{proof}
We will prove that the tropical linear space $L_{A,u}^h$ contains the cone $C_\tau$ in the coordinate subspace $x_{n+1}=0$. This implies the theorem by \rev{Proposition \ref{proposition: homogenization}}. Let $\tau\in M(A)$ and $\sigma=[n]\setminus \tau$ and consider the set
\begin{equation}\label{eq: definition set C}
C = \big\lbrace(x,0)\in\mathbb{R}^{n+1}\,:\, (x,0)+\sum_{i\in\sigma}\lambda_ie_i\in L_{A,u}^h\text{~for all~}\lambda\geq 0\big\rbrace.
\end{equation}
Write $\hat{x}=(x,0)\in C$ and consider the quantity
\begin{equation}\label{eq: sum for tropical linear space}
S_\gamma = \langle e_\gamma, \hat{x}+\sum_{i\in\sigma}\lambda_ie_i\rangle-\pi_{\gamma}=\langle e_\gamma,\hat{x}\rangle-\pi_{\gamma}+\sum_{i\in\gamma\cap\sigma}\lambda_i,
\end{equation} 
for some basis $\gamma$  of $M(B^h)$. If $\lambda\gg0$ is chosen large enough, then $S_{\gamma}$ in \eqref{eq: sum for tropical linear space} will always be larger for bases with the largest possible intersection $\gamma\cap\sigma$ compared to other bases. These bases of maximal intersection are of the form $\gamma=\sigma+j$ with $j\in\tau\cup(n+1)$. Since $\hat{x}+\sum_{i\in\sigma}\lambda_ie_i\in L_{A,u}^h$ by definition of the set $C$, \rev{ Lemma \ref{lemma: tropical linear space initial matroid} applied to $L_{A,u}^h$ implies} that the union of all bases $\gamma$ for which $S_\gamma$ is maximal must equal $[n+1]$. This implies that $S_{\sigma+j}$ must be maximal for all $j\in\tau\cup(n+1)$\rev{---since otherwise $j$ would not be covered by a maximal basis---}, and thus equal. In particular, since $\pi_{\sigma+(n+1)}=0$, we find
$$
S_{\sigma+j}=S_{\sigma+(n+1)}\iff \langle e_{\sigma+j},\hat{x}+\sum_{i\in\sigma}\lambda_i e_i\rangle-\pi_{\sigma+j} = \langle e_{\sigma+(n+1)},\hat{x}+\sum_{i\in\sigma}\lambda_i e_i\rangle\iff x_j=\pi_{\sigma+j}=:w^{(\tau)}_j,
$$
for all $j\in\tau$ and for every $(x,0)\in C$. \rev{The last equality $\pi_{\sigma+j} = w_j^{(\tau)}$ follows from Proposition \ref{proposition: tropical plucker vector} and matroid duality.}

The tropical linear space $L_{A,u}^h$ is a polyhedral complex dual to the regular subdivision of the matroid polytope $\conv(e_{\gamma}:\gamma\in M(B^h))$ induced by the weight vector $\pi$, restricted to the faces that correspond to loopless matroids; see for instance \cite[Lemma 4.4.7]{maclagan_2015_introduction}. As a result, two points  in the same cell of $L_{A,u}^h$ maximize $S_\gamma$ over the same collection of bases $\gamma$ of $M(B^h)$, and any point  in a lower-dimensional face of this cell will maximize $S_\gamma$ over the same bases and potentially others. For any $(x,0)\in C$ we can always choose $\lambda\gg0$ such that $(x',0)=(x,0)+\sum_{i\in\sigma}\lambda_ie_i$ maximizes $S_\gamma$ \rev{on the bases} $\{\sigma+j \, : \, j\in\tau\cup(n+1)\}$ \rev{among all bases of $M(B^h)$}. Since the points $(x,0),(x',0)$ are in the same cell, with $(x,0)$ potentially being in a lower-dimensional face, we know that $(x,0)$ also maximizes $S_\gamma$ over the same bases and thus in particular
$$
S_\gamma\leq S_{\sigma+(n+1)}\iff \langle e_\gamma,(x,0)\rangle-\pi_{\gamma}\leq \langle e_\sigma,x\rangle \text{~for all $\gamma\in M(B^h)$ and all $(x,0)\in C$}.
$$
By Lemmas \ref{claim: one-dimensional inequalities imply all others} and \ref{claim: expression for vertex} below, we know that this is equivalent to
$$
x_i\geq \max\lbrace\pi_{\sigma+j}\,:\,j\in\tau\text{~s.t.~}\sigma+j-i\in M(B) \rbrace =: w^{(\tau)}_i \text{~for all $i\in\sigma$}.
$$
The set $C$ can thus be characterized as
\begin{align*}
C &= \big\lbrace (x,0)\in\mathbb{R}^n\,:\, x_j=w^{(\tau)}_j\text{~and~}x_i\geq w^{(\tau)}_i\text{~for all~}j\in\tau,i\in\sigma\big\rbrace
\\
&=\big\lbrace (w^{(\tau)},0)+\sum_{i\in\sigma}\lambda_ie_i\text{~for~}\lambda\geq 0\big\rbrace.
\end{align*}
This is a cone with vertex $w^{(\tau)}$ and rays $(e_i)_{i\in\sigma}$ and thus must lie in a top-dimensional cone of $L_{A,u}^h$ with vertex $(\tilde{w},0)$ and rays $(e_i)_{i\in\sigma}$. The definition of $C$ in \eqref{eq: definition set C} implies that $(\tilde{w},0)$ must lie in $C$ as well, and thus that $\tilde{w}=w^{(\tau)}$. This shows that $w^{(\tau)}$ is a vertex in $L_{A,u}^h$. Dehomogenizing $L_{A,u}^h$ to $L_{A,u}$ by forgetting the $(n+1)$st coordinate completes the proof.
\end{proof}
We now prove the two technical lemmas.
\begin{lemma}\label{claim: one-dimensional inequalities imply all others}
Fix $\tau\in M(A)$ and let $(x,0)\in\mathbb{R}^{n+1}$ with $x_j=\pi_{\sigma+j}$ for all $j\in\tau$. Then the inequality $\langle e_\gamma, (x,0)\rangle-\pi_\gamma\leq \langle e_\sigma,(x,0)\rangle$ holds for all $\gamma\in M(B^h)$  if and only if it holds for all $\gamma\in M(B^h)$ with $\vert\sigma\backslash\gamma\vert=1$
\rev{where $\sigma = [n] \setminus \tau$}.
\end{lemma}
\begin{proof}
The ``only if'' direction is immediate; we show the ``if'' direction. We use the abbreviation $\hat{x}=(x,0)$. Let $\gamma\in M(B^h)$ be a basis such that $\vert\sigma\backslash\gamma\vert=1$. This basis can always be written as $\gamma=\sigma+j+k-i$ for some $i\in\sigma$ and $j,k\in\tau\cup(n+1)$ and, moreover, by the basis exchange axiom and the earlier observation that $M(B)$ has no coloops, such an exchange exists for every element $i\in\sigma$. We now rewrite the inequalities from the claim as
\begin{align}
&\langle e_\sigma,\hat{x}\rangle \geq \langle e_\gamma,\hat{x}\rangle-\pi_\gamma &\text{~$\forall\gamma\in M(B^h)$ s.t. $\vert\sigma\backslash\gamma\vert=1$}\nonumber
\\
\iff& \langle e_\sigma,\hat{x}\rangle \geq \langle e_\sigma,\hat{x}\rangle -\pi_{\sigma+j+k-i}+\pi_{\sigma+j}+\pi_{\sigma+k}-\hat{x}_i &\text{~$\forall \sigma+j+k-i\in M(B^h)$}\nonumber
\\
\iff& \hat{x}_i \geq \pi_{\sigma+j} + \pi_{\sigma+k} - \pi_{\sigma+j+k-i}\nonumber
\\
\Longrightarrow~&
\hat{x}_i \geq \max\lbrace \pi_{\sigma+j}\,:\, j\in\tau\text{~s.t.~} \sigma+j-i\in M(B)\rbrace &\text{~$\forall i\in\sigma$},\label{eq:bound for wi}
\end{align}
where the last implication follows from choosing $k=(n+1)$ and then maximizing over all possible choices of $j$. 

We now show that the inequalities \eqref{eq:bound for wi} for $\hat{x}_i$ imply the inequalities from the claim for all bases of $M(B^h)$. For any basis $\gamma\in M(B^h)$ with $k=\vert\sigma\backslash\gamma\vert\geq 2$, we can write
$$
\gamma = \sigma+j_1+\dots+j_k+j_{k+1} - i_1 - \dots - i_k
$$
with $\{j_1,\dots,j_{k+1} \}\subseteq \tau\cup(n+1)$ and $\{i_1,\dots,i_k \} \subseteq \sigma$ all distinct. In order to show that
\begin{align*}
\text{\eqref{eq:bound for wi}}&\,\,\,\Longrightarrow \langle e_\sigma,\hat{x}\rangle\geq \langle e_\gamma,\hat{x}\rangle - \pi_\gamma  &\forall\gamma\in M(B^h)
\end{align*}
we need to prove 
\begin{align*}
\text{\eqref{eq:bound for wi}}&\,\,\,\Longrightarrow\underbrace{\hat{x}_{i_1} + \dots + {\hat{x}}_{i_{k}}}_{(i)} \geq \underbrace{\pi_{\sigma+j_1}+\dots+\pi_{\sigma+j_{k+1}} - \pi_{\gamma}}_{(ii)} &\forall\gamma\in M(B^h).
\end{align*}
First, using \eqref{eq:bound for wi} we can bound the sum $(i)$ by
$$
(i)=\sum_{r\in [k]} \hat{x}_{i_r} \geq \sum_{r\in[k]}\max\lbrace \pi_{\sigma+j}\,:\, j\in\tau\text{~s.t.~}\sigma+j-i_{r}\in M(B)\rbrace.
$$
We will now show that this is an upper bound for $(ii)$. Recall that $\pi_{\gamma}$ is the following minimum:
$$
\pi_{\gamma} = \min\lbrace w_z \,:\, z\in\gamma\text{~s.t.~} \sigma+j_1+\dots+j_{k+1}-i_1-\dots-i_k-z\in M(B)\rbrace.
$$
There are two cases for the index $z$ at which the minimum is achieved: either $z=j_{k+1}$ (or, equivalently, any other element in $j_1,\dots,j_{k+1}$), or $z\in\sigma\setminus\lbrace i_1,\dots,i_k\rbrace$.
\\
\textbf{Case 1}: $\pi_\gamma$ achieves its minimum at $j_{k+1}$. First, we note that this implies that
$$
\pi_{\gamma} = w_{j_{k+1}} \geq \min\lbrace w_{j_{k+1}} \,,\, w_r \,:\, r\in\sigma \text{~s.t.~}\sigma+j_{k+1}-r\in M(B)\rbrace = \pi_{\sigma+j_{k+1}},
$$
which further implies that
$$
(ii) =\pi_{\sigma+j_1}+\dots+\pi_{\sigma+j_k}+(\pi_{\sigma+j_{k+1}}-\pi_{\gamma})\leq \pi_{\sigma+j_1}+\dots+\pi_{\sigma + {j_k}}.
$$
Second, if the minimum in $\pi_\gamma$ is achieved at $j_{k+1}$ then we must have $\gamma-j_{k+1}\in M(B)$ by definition of $\pi_{\gamma}$. By the symmetric basis exchange property applied to $\gamma-j_{k+1}$ and $\sigma$, there exists an index in $\lbrace j_1,\dots,j_k\rbrace$, say $j_k$, such that both $\gamma_k:=\gamma-j_{k+1}-j_k+i_k\in M(B)$ and $\sigma_k:=\sigma+j_k-i_k\in M(B)$. We repeat the symmetric basis exchange, now applied to $\gamma_k$ and $\sigma$. This says that there exists an index in $\lbrace j_1,\dots,j_{k-1}\rbrace$, say $j_{k-1}$, such that both $\gamma_{k-1}:=\gamma_{k}-j_{k-1}+i_{k-1}\in M(B)$ and $\sigma_{k-1}:=\sigma+j_{k-1}-i_{k-1}\in M(B)$. We repeat this until we obtain $\gamma_1=\sigma$, at which point we have shown
$$
\sigma_r = \sigma + j_{r} - i_{r}\in M(B) \text{~for all $r\in[k]$.}
$$
Taking these conditions on $j_1,\dots,j_k$ into account, we can further bound the sum $(ii)$ as 
$$
(ii)\leq \pi_{\sigma+j_1}+\dots + \pi_{\sigma+j_{k}} \leq \sum_{r\in[k]}\max\lbrace\pi_{\sigma+j}\,:\,j\in\tau\text{~s.t.~}\sigma+j-i_r\in M(B) \rbrace\leq (i).
$$
This completes the proof for the first case.
\\
\textbf{Case 2:} $\pi_{\gamma}$ achieves its minimum at $z\in\sigma\backslash \lbrace i_1,\dots,i_k\rbrace$. First, this implies that $\gamma-z\in M(B)$. By the symmetric basis exchange property applied to $\gamma-z$ and $\sigma$, we know that there exists an index in $\lbrace j_1,\dots,j_{k+1}\rbrace$, say $j_{k+1}$, such that both $\gamma-j_{k+1}\in M(B)$ and $\sigma+j_{k+1}-z\in M(B)$. The latter implies that
$$
\pi_{\sigma+j_{k+1}}=\min\lbrace w_r \,:\, r\in\sigma+j_{k+1}\text{~s.t.~} \sigma+j_{k+1}-r\in M(B)\rbrace\leq  w_z=\pi_{\gamma},
$$
and thus that
$$
(ii)=\pi_{\sigma+j_1}+\dots+\pi_{\sigma+j_k}+(\pi_{\sigma+j_{k+1}}-\pi_{\gamma})\leq \pi_{\sigma+j_1}+\dots+\pi_{\sigma+j_{k}}.
$$
Second, as in case 1 above, we can use the fact that $\gamma-j_{k+1}\in M(B)$ and repeated symmetric basis exchanges with $\sigma$ to show that $\sigma+j_{r}-i_r\in M(B)$ for all $r\in[k]$. This again implies that
$$
(i) \leq \sum_{r\in [k]}\max\lbrace \pi_{\sigma+j}\,:\, j\in\tau\text{~s.t.~} \sigma+j-i_r\in M(B)\rbrace \leq (ii).
$$
This completes the proof of the second case, and thus the proof of the claim.
\end{proof}
\begin{lemma}\label{claim: expression for vertex}
Fix $\tau\in M(A)$ and let $(x,0)\in\mathbb{R}^{n+1}$ with $x_j=\pi_{\sigma+j}$ for all $j\in\tau$. Then
the inequality $\langle e_\gamma,(x,0)\rangle-\pi_\gamma\leq\langle e_\sigma,(x,0)\rangle$ holds for all $\gamma\in M(B^h)$ with $\vert\sigma\backslash\gamma\vert=1$ if and only if $x_i\geq \max\lbrace \pi_{\sigma+j}:j\in\tau\text{~s.t.~}\sigma+j-i\in M(B)\rbrace$ for all $i\in\sigma$.
\end{lemma}
\begin{proof}
We write $\hat{x}=(x,0)\in\R^{n+1}$. As in the proof of Lemma \ref{claim: one-dimensional inequalities imply all others}, we know that 
\begin{align*}
&\langle e_\sigma,\hat{x}\rangle\geq\langle e_\gamma,\hat{x}\rangle-\pi_{\gamma} &\text{~$\forall\gamma\in M(B^h)$}
\\
\iff & \hat{x}_i\geq \pi_{\sigma+j}+\pi_{\sigma+k}-\pi_{\sigma+j+k-i} &\text{~$\forall\sigma+j+k-i\in M(B^h)$}
\\
\iff & \hat{x}_i\geq \underbrace{\max\lbrace \pi_{\sigma+j}+\pi_{\sigma+k}-\pi_{\sigma+j+k-i}\,:\,j,k\text{~s.t.~}\sigma+j+k-i\in M(B^h)\rbrace}_{(\star)} &\text{~$\forall i\in\sigma$}.
\end{align*}
We now show that $(\star)=\max\lbrace \pi_{\sigma+j}\,:\,j\in\tau\cup(n+1)\text{~s.t.~}\sigma+j-i\in M(B)\rbrace$ to prove the claim. First, by choosing $k=(n+1)$ such that $\pi_{\sigma+k}=\pi_{\sigma+j+k-i}=0$, we find that
\begin{equation}\label{eq: lower bound for star}
(\star)\geq \max\lbrace \pi_{\sigma+j}\,:\,j\in\tau\text{~s.t.~}\sigma+j-i\in M(B)\rbrace.
\end{equation}
Second, assume that the maximum in $(\star)$ is achieved for some $j',k'$ and recall that 
$$
\pi_{\sigma+j'+k'-i} = \min\lbrace w_z\,:\, z\in\sigma+j'+k'\text{~s.t.~}\sigma+j'+k'-i-z\in M(B)\rbrace.
$$
There are two cases for the index $z$ at which the minimum is achieved: either $z=k'$ (which is equivalent to $z=j'$) or $z\in\sigma-i$.\\
\textbf{Case 1:} $\pi_{\sigma+j'+k'-i}$ achieves its minimum at $k'$. First, this immediately implies that 
$$
\pi_{\sigma+j'+k'-i} = w_{k'} \geq \min\lbrace w_{k'}\,,\, w_r : r\in\sigma\text{~s.t.~} \sigma+k'-r\in M(B)\rbrace = \pi_{\sigma+k'}
$$
and thus that
$$
(\star)=\pi_{\sigma+j'}+(\pi_{\sigma+k'}-\pi_{\sigma+j'+k'-i})\leq \pi_{\sigma+j'}.
$$
Furthermore, since the minimum is achieved at $k'$, this implies that $\sigma+j'-i\in M(B)$ and thus that
$$
(\star)\leq\pi_{\sigma+j'}\leq \max\lbrace \pi_{\sigma+j}\,:\,j\in\tau\text{~s.t.~}\sigma+j-i\in M(B)\rbrace.
$$
Together with \eqref{eq: lower bound for star}, this implies that $(\star)=\max\lbrace \pi_{\sigma+j}:j\in\tau\text{~s.t.~}\sigma+j-i\in M(B)\rbrace$, which proves the claim in this case.
\\
\textbf{Case 2:} $\pi_{\sigma+j'+k'-i}$ achieves its minimum at $z\in\sigma-i$. 
This implies that $\sigma+j'+k'-i-z\in M(B)$. By the symmetric basis exchange property applied to this basis and $\sigma$, we find that there exists an index in $\lbrace j',k'\rbrace$, say $k'$, such that both $\sigma+k'-z\in M(B)$ and $\sigma+j'-i \in M(B)$. Since $\sigma+k'-z\in M(B)$ we find that
$$
\pi_{\sigma+k'}=\min\lbrace w_r \,:\,r\in\sigma+k'\text{~s.t.~}\sigma+k'-r\in M(B)\rbrace\leq w_z =\pi_{\sigma+j'+k'-i}
$$
and thus that
$$
(\star)=\pi_{\sigma+j'}+(\pi_{\sigma+k'}-\pi_{\sigma+j'+k'-i})\leq \pi_{\sigma+j'}.
$$
Furthermore, because $\sigma+j'-i\in M(B)$ holds we can further write
$$
(\star)\leq\pi_{\sigma+j'}\leq\max\lbrace\pi_{\sigma+j}\,:\,j\in\tau\text{~s.t.~}\sigma+j-i\in M(B)\rbrace.
$$
Together with \eqref{eq: lower bound for star}, this implies that $(\star)=\max\lbrace \pi_{\sigma+j}:j\in\tau\text{~s.t.~}\sigma+j-i\in M(B)\rbrace$ also for this case, and this completes the proof of the claim.
\end{proof}
\begin{example}[Hirzebruch surface] \label{ex: Hirzebruch}
We give a detailed example that illustrates the computation of $w^{(\tau)}$ as defined in Theorem \ref{theorem: cones in LAu}. 
We consider a Hirzebruch surface with data
$$
A = \begin{pmatrix}
1&1&1&1&1&1\\
0&1&0&1&2&3\\
0&0&1&1&1&1
\end{pmatrix}
\quad\text{~and~}\quad w=(6,8,7,6,4,0)\quad\quad\quad \raisebox{-0.35\height}{\includegraphics[width=0.13\textwidth]{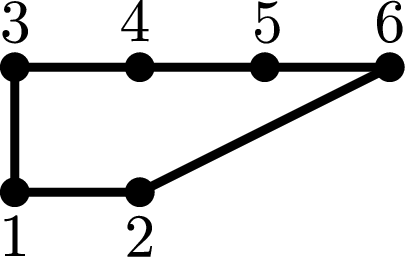}}
$$ 
The matroid $M(A)$ is a rank $3$ matroid on $[6]$ with $16$ bases out of the possible $20=\binom{6}{3}$. The four non-bases are $\lbrace 345,346,356,456\rbrace$ and correspond to the collinear triples of points \rev{among $\{a_i\}_{i\in [n]}$} in $Q_A$. We first consider the basis $256\in M(A)$ and compute the vertex $w^{(256)}$:
\begin{align*}
 w_2^{(256)} &= \min \{w_2, w_1,\cancel{w_3},\cancel{w_4}\} = 6 
 & w_1^{(256)} &= \max \{w_2^{(126)}, w_5^{(126)}, w_6^{(126)}\} = 6 \\
 w_5^{(256)} &= \min\{w_5, w_1,w_3,w_4\} = 4 & w_3^{(256)} &= \max\{\cancel{w_2^{(126)}}, w_5^{(126)}, w_6^{(126)}\} = 4 \\ 
 w_6^{(256)} &= \min\{w_6, w_1,w_3,w_4\} = 0 & w_4^{(256)} &= \max\{\cancel{w_2^{(126)}}, w_5^{(126)}, w_6^{(126)}\} = 4. 
\end{align*}
We briefly explain the terms that do not appear in the minimization or maximization. Recall that for $j\in 256$, $w^{(256)}_j$ is the minimum of $w_j$ and all $w_i$ with $i\in 134$ such that $256-j+i\in M(A)$. For $w^{(256)}_2$ we have that $256-2+3=356$ is not a basis of $M(A)$, hence $w_3$ does not appear in this minimization, and similarly $256-2+4=456$ is not a basis of $M(A)$ and thus $w_4$ does not appear. For $i\in 134$, $w^{(256)}_i$ is the maximum of $w_j^{(256)}$ for all $j\in 256$ such that $256-j+i\in M(A)$. For $w_3^{(256)}$ we have that $256-2+3=356$ is not a basis of $M(A)$ and thus $w_2^{(126)}$ does not appear in the maximization. For the same reason, $w_2^{(126)}$ does not appear in the maximization formula for $w_4^{(256)}$. We conclude that $w^{(256)} = (6,\, 6,\, 4,\, 4,\, 4,\, 0)$. 

For the basis $126\in M(A)$, we compute the vertex $w^{(126)}$ as:
\begin{align*}
 w_1^{(126)} &= \min(w_1, w_3,w_4,w_5) = 4 & w_3^{(126)} &= \max(w_1^{(126)}, w_2^{(126)}, w_6^{(126)}) = 4 \\
 w_2^{(126)} &= \min(w_2, w_3,w_4,w_5) = 4 & w_4^{(126)} &= \max(w_1^{(126)}, w_3^{(126)}, w_6^{(126)}) = 4 \\
 w_6^{(126)} &= \min(w_6, w_3,w_4,w_5) = 0 & w_5^{(126)} &= \max(w_1^{(126)}, w_2^{(126)}, w_6^{(126)}) = 4. 
\end{align*}
We conclude that $w^{(126)} = (4,\, 4,\, 4,\, 4,\, 4,\, 0)$.
\end{example}

\begin{example}\label{example: square with central point}
We give an example to illustrate that some vertices in the tropical affine space $L_{A,u}$ may not appear as $w^{(\tau)}$ corresponding to a basis $\tau\in M(A)$.
For this, we consider the tropical affine space $L_{A,u}$ defined by 
$$
A = \begin{pmatrix}
1&1&1&1&1\\
1&0&2&2&0\\
1&0&0&2&2
\end{pmatrix}\quad\text{~and~}\quad w=(0,1,2,3,4)\quad\quad\quad\raisebox{-0.35\height}{\includegraphics[width=0.09\textwidth]{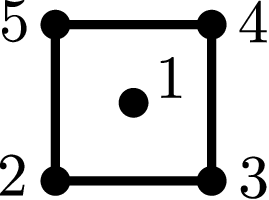}}
$$
The matroid $M(A)$ has $8$ bases. These are all \rev{the} $3$-element subsets of $[5]$ except $124$ and $135$. The tropical affine space $L_{A,u}$
is a pure two-dimensional polyhedral complex with four vertices
$$ v_0 = (0, \, 0, \, 0, \, 0, \,0), \quad  v_1 = (0, \, 1, \, 2, \, 1, \,2), \quad 
v_2 = (0, \, 1, \, 0, \, 1, \,0), \quad
v_3 = (0, \, 0, \, 2, \, 0, \,2).$$
The maximal faces consist of the bounded 
face $\conv(v_0,v_1,v_2,v_3)$ and the six 
faces $\conv(v_0,v_2) + \pos(e_1)$, $\conv(v_0,v_3) + \pos(e_1)$, $\conv(v_1,v_2) + \pos(e_2)$, $\conv(v_1,v_2) + \pos(e_4)$, $\conv(v_1,v_3) + \pos(e_3)$, and $\conv(v_1, v_3) + \pos(e_5)$ together
with the eight cones $C_\tau$ corresponding to the bases of $M(A)$. The vertices of these
eight cones are 
\begin{align*}
v_1 &=w^{(125)}=w^{(145)}=w^{(134)}=w^{(123)},\\
v_2 &=w^{(235)}=w^{(345)}, \\
v_3 &=w^{(245)}=w^{(234)}.
\end{align*}
The fourth vertex $v_0 = (0,0,0,0,0)$ does not appear as $w^{(\tau)}$ for any $\tau\in M(A)$.
\end{example}

\subsection{The $\tau$-operator}
To conclude this section, we spend a bit more time on the expression of the vertices $w^{(\tau)}$. To this end, we study properties of the endomorphism of $\R^n$ that maps a tropical data vector $w$ to $w^{(\tau)}$. 
\begin{definition}
Let $M$ be a matroid on $[n]$ and let $\tau$ be a basis of $M$. The $\tau$-operator is the map $(\,\cdot\,)^{(\tau)}:\R^n\rightarrow\R^n$ where
$$
x\mapsto \begin{cases}
x^{(\tau)}_j = \min\lbrace x_j\,,\, x_i\,:\, i\in [n]\setminus\tau\text{~s.t.~}\tau-j+i\in M \rbrace &\textup{~for all $j\in\tau$}\\
x^{(\tau)}_i = \max\lbrace x_j^{(\tau)} \,:\, j\in\tau\text{~s.t.~}\tau-j+i\in M\rbrace &\textup{~for all $i\in[n]\backslash\tau$}.
\end{cases}
$$
\end{definition}
\begin{proposition}
The $\tau$-operator is 
\begin{itemize}
    \item[i)] piecewise linear and continuous,
    \item[ii)] shift invariant: $(x+\alpha\cdot \mathbbm{1})^{(\tau)}=x^{(\tau)}+\alpha\cdot \mathbbm{1},$
    \item[iii)] contractive: $\min \{x_i \, : \, i \in \tau\}  \leq x^{(\tau)}_{\rev{j}}\leq x_{\rev{j}}$ \rev{for all $j\in[n]$}, and 
    \item[iv)] idempotent: $(x^{(\tau)})^{(\tau)}=x^{(\tau)}$
\end{itemize}
for all $x\in \R^n$ and $\alpha\in\R$.
\end{proposition}
\begin{proof}
\emph{i)} This follows because the $\tau$-operator is linear on the closed cones defined by the inequalities $x_{\mu(1)} \leq x_{\mu(2)} \leq \cdots \leq x_{\mu(n)}$ for each permutation $\mu$ of $[n]$. 

\emph{ii)} Follows directly from the definition. By this invariance, the $\tau$-operator can also be defined as an endomorphism of $\R^n/\R \mathbbm{1}$.

\emph{iii)} By shift invariance we assume $x\geq 0$. The entries of $x^{(\tau)}$ are a subset of the entries of $x$ and thus $x\geq 0$ implies $x^{(\tau)}\geq 0$. Also, for $j\in\tau$ we have

$$x^{(\tau)}_j=\min\lbrace x_j\,,\, x_i\,:\, i\in [n]\setminus\tau\text{~s.t.~}\tau-j+i\in M\rbrace\leq x_j.$$
Now let $i\in\sigma \rev{= [n] \setminus \tau} $ and define $\mathcal{J}_i=\lbrace j\in\tau\,:\, \tau-j+i\in M\rbrace$. We have $x^{(\tau)}_j\leq x_i$ for all $j\in\mathcal{J}_i$ and thus
$x^{(\tau)}_i = \max\lbrace x^{(\tau)}_j\,:\, j\in\mathcal{J}_i\rbrace\leq x_i$.
This implies $x^{(\tau)}_{\rev{j}}\leq x_{\rev{j}}$ \rev{for all $j$} as required.

\emph{iv)} Let $x\in\mathbb{R}^n$, $y=x^{(\tau)}$ and $z=y^{(\tau)}=(x^{(\tau)})^{(\tau)}$. To show that $z=y$, it suffices to prove that $z_j=y_j$ for all $j\in\tau$. By contractivity we know that $z_j\leq y_j$. Fix $j\in\tau$ and let $\mathcal{I}_j=\lbrace i\in[n]\setminus\tau\,:\,\tau-j+i\in M\rbrace$. For all $i\in\mathcal{I}_j$ we have
$$
y_i=\max\lbrace y_k\,:\, k\in\tau\text{~s.t.~}\tau-k+i\in M\rbrace\geq y_j.
$$
Thus we find
$z_j=\min\lbrace y_i:i\in\mathcal{I}_j\rbrace\geq y_j$
which completes the proof.
\end{proof}
\begin{proposition}
The image of the cone $C_\tau\subseteq L_{A,u}$ 
under the $\tau$-operator
is its vertex $w^{(\tau)}$.
\end{proposition}
\begin{proof}
By \rev{contractivity} we find for all $j\in\tau$ that
$$
w^{(\tau)}_j\rev{\leq\min\{w_j, w_i^{(\tau)}:i\in\sigma\text{~s.t.~}\sigma+j-i\in M(B)\}} \leq w^{(\tau)}_i \qquad\forall i\in\sigma\text{~s.t.~}\sigma+j-i\in M(B).
$$
Let $y=w^{(\tau)}+\sum_{i\in\sigma}\lambda_ie_i\in C_\tau$ with $\lambda\geq 0$. We then find again for all $j\in\tau$ that
$$
y_j=w^{(\tau)}_j\leq w^{(\tau)}_i\leq y_i\qquad\forall i\in\sigma\text{~s.t.~}\sigma+j-i\in M(B).
$$
This implies $y^{(\tau)}_j = y_j = w^{(\tau)}_j$ for all $j\in\tau$ and thus $y^{(\tau)}=w^{(\tau)}$, which completes the proof.
\end{proof}
\begin{question}
The underlying matroid for the tropical affine space $L_{A,u}$ is representable. It would be good to know how much is true otherwise. For instance, does the function $\pi$ in Proposition \ref{proposition: tropical plucker vector} satisfy the tropical Pl\"{u}cker relations \cite[Equation 4.4.2]{maclagan_2015_introduction} if $M$ is not representable? If the answer is positive, then $w\in\R^n$ induces a tropical affine space $L$ associated with any matroid $M$. Is $w^{(\tau)}$ a vertex of $L$ in this case? 
\end{question}

\section{Subdivisions and tropical critical 
points} \label{section: subdivisions and tropical critical points}
We now apply our results on $L_{A,u}$ to the tropical toric maximum likelihood estimation problem. We study the intersection of $\row(A)$ with the cones $C_\tau =w^{(\tau)}+\pos(e_i\mid i\in[n]\backslash\tau)$ which, if this intersection occurs, contributes a tropical critical point that is a linear function of the vertex $w^{(\tau)}$. Furthermore, the bases $\tau$ for which $C_\tau$ contributes a tropical critical point, and the multiplicity of this point, can be related to the geometry of $Q_A$ by certain subdivisions. With these results, we find explicit expressions for the tropical critical points in terms of the tropical data vector in a number of general situations.

We start with the relevant background on polyhedral subdivisions. For $A=(a_1~\cdots~a_n)$ with columns indexed by $[n]$, we write $\conv_A(\sigma):=\conv(a_i\mid i\in\sigma)$ for $\sigma\subseteq[n]$. We say that $\tau\subseteq \sigma$ is a \emph{face} of $\sigma$ if there exists a linear form $\psi$ on $\R^k$ such that $(\psi(a_i))_{i\in\sigma}$ is maximized precisely for $i\in\tau$.

\begin{definition}[polyhedral subdivision]
Let $A\in\Z^{k\times n}$ and $Q_A=\conv_A([n])\subseteq \R^k$. A \textit{polyhedral subdivision} $\Delta$ of $Q_A$ is a collection of subsets $\sigma$ of $[n]$, which are called \emph{cells} (and are identified with $\conv_A(\sigma)$),
that satisfy:
\begin{itemize}
    \item {Closure:} If $\sigma\in\Delta$ and $\tau$ is a face of $\sigma$, then $\tau\in\Delta$. 
    \item {Union:} $\bigcup_{\sigma\in\Delta}\conv_A(\sigma) = Q_A$.
    \item {Intersection:} If $\sigma_1,\sigma_2\in\Delta$ distinct, then $\operatorname{relint}(\conv_A(\sigma_1))\cap\operatorname{relint}(\conv_A(\sigma_2))=\emptyset$.
\end{itemize}
\end{definition}
All subdivisions in this article are polyhedral so we will omit the adjective. A subdivision is called a \textit{triangulation} if all cells are simplices, i.e., if $\lbrace a_i\rbrace_{i\in\sigma}$ is affinely independent for every cell $\sigma$. We say that a subset $\tau\subseteq [n]$ \textit{lies in a cell} of $\Delta$ if there exists a cell $\sigma$ such that $\tau\subseteq\sigma$. A subdivision $\Delta'$ of $Q_A$ \textit{refines} another subdivision $\Delta$ if every cell of $\Delta'$ lies in a cell of $\Delta$. A cell in a subdivision is called \emph{maximal} if it is maximal with respect to the subset partial order and the \textit{volume} of a maximal simplex $\tau$ is $\vol_A(\tau):=\vol(\conv_A(\tau))=\vert\det(A_\tau)\vert$. We note that the maximal simplices in a triangulation of $Q_A$ are bases of the matroid $M(A)$.

The following class of subdivisions will be particularly relevant. A \emph{regular subdivision} $\Delta_{\omega}$ induced by weights $(\omega_i)_{i\in[n]}$ is a subdivision whose cells are given by the lower faces of the convex hull of the lifted points $\lbrace(a_i,\omega_i)\rbrace_{i\in[n]}\subseteq\R^{n+1}$. The following theorem can serve as a definition.
\begin{theorem}[{\cite[Theorem 2.3.20, Corollary 5.2.7]{deloera_2010_triangulations}}]\label{theorem: definition of regular subdivision}
The regular subdivision $\Delta_{\omega}$ of $Q_A$ induced by the weight vector $\omega$ satisfies the entrywise inequality
$$
A^T(A_\tau^T)^{-1}\omega_{\tau}\leq \omega
$$
for all simplices $\tau$ that lie in a cell of $\Delta_{\omega}$. Equality holds for the $i$th entry if and only if $\tau+i$ lies in a cell of $\Delta_{\omega}$.
\end{theorem}
We now connect the tropical intersection problem to regular subdivisions of $Q_A$. 

\ConeIntersection*
\begin{proof}
The linear subspace $\row(A)$ and the affine span of $C_\tau$ have complementary dimension. \rev{Since $A_\tau$ is invertible, these subspaces together span $\R^n$ and thus}  intersect in precisely one point ${q}=A^Tx$. Because $A_\tau$ is invertible, this point is determined by
$$
A_\tau^T x = w^{(\tau)}_\tau\iff x=(A_\tau^T)^{-1}w^{(\tau)}_\tau\iff{q} = A^T(A_\tau^T)^{-1}w^{(\tau)}_{\tau}.
$$
It remains to check whether this point lies in $C_\tau$. This is equivalent to ${q}\geq w^{(\tau)}$ and thus $A^T(A^T_{\tau})^{-1}w^{(\tau)}_{\tau}\geq w^{(\tau)}$ \rev{, where equality for the entries in $\tau$ follows by construction}. This shows that $\row(A)$ intersects $C_\tau$ if and only if $\tau$ lies in a cell of the regular subdivision $\Delta_{-w^{(\tau)}}$. For the multiplicity of ${q}$, let $\sigma=\lbrace i_1,\dots,i_{n-k}\rbrace$ and compute
$$\operatorname{mult}(q) \, = \, \big\vert \det\big(
A^T \mid e_{i_1}\mid \dots \mid e_{i_{n-k}}
\big)\big\vert \, = \,  \big\vert\det\big(A_{[n]\setminus \sigma}^T\big) \big\vert \, = \, \vol_A(\tau).$$
\end{proof}
\begin{example}[Hirzebruch surface revisited]
To illustrate Proposition \ref{proposition:subdivision condition for cone intersection} we return to the Hirzebruch surface of Example \ref{ex: Hirzebruch}. For tropical data vector $w=(6,8,7,6,4,0)$, we computed that the
cone $C_{256}$ has vertex $w^{(256)} = (6,6,4,4,4,0)$. The maximal cells of the subdivision $\Delta_{-w^{(256)}}$ are $12345$ and $256$; in particular $256$ lies in a cell of this subdivision. Therefore,
$$
\hat{q}(256) = A^T(A_{256}^T)^{-1}(6,4,0)^T = (10,6,12,8,4,0)^T = w^{(256)} + 4e_1 + 8e_3 + 4e_4
$$
is a tropical critical point with 
multiplicity 
$$
\det \begin{pmatrix}
1&1&1&1&1&1\\
0&1&0&1&2&3\\
0&0&1&1&1&1 \\
1&0&0&0&0&0 \\
0&0&1&0&0&0 \\
0&0&0&1&0&0
\end{pmatrix} = 1.
$$
In contrast, the cone $C_{126}$ has vertex
$w^{(126)} = (4,4,4,4,4,0)$ and 
the maximal cells of the subdivision $\Delta_{-w^{(126)}}$ are again $12345$ and $256$. Since $126$ does not lie in any of these cells,
$C_{126}$ and $\row(A)$ do not intersect. Indeed
$$A^T(A_{126}^T)^{-1}(4,4,0)^T = (4,4,0,0,0,0)^T = w^{(126)} - 4e_3 - 4e_4 - 4e_5.$$
\end{example}
We recall that the likelihood function has $\deg(cX_A)=\vol_A(Q_A)$ critical points on $cX_A$ and thus $\vol_A(Q_A)$ tropical critical points, counted with multiplicity. The following theorem gives a sufficient condition for $A$ and $w$ that guarantees that all intersection points can be found on the cones $C_\tau$.

\ConditionForAllIntersections*
\begin{proof}
Let $v\in\R^n$ induce the regular triangulation $\Delta = \Delta_v$ of $Q_A$ and let $\tau\in\Delta$ be a maximal cell. Since $\tau$  lies in a cell of $\Delta_{-w^{(\tau)}}$, it lies in a cell of $\Delta_{-w^{(\tau)}+\varepsilon v}$ for sufficiently small $\varepsilon$. As a result, we have
$$
A^T(A_\tau^T)^{-1}(w^{(\tau)}-\varepsilon v)_{\tau} \geq (w^{(\tau)}-\varepsilon v).
$$
By the same argument as in the proof of Proposition \ref{proposition:subdivision condition for cone intersection}, this implies that $\row(A)$ intersects $C_{\tau}-\varepsilon v$ in the point $\hat{q}_{\varepsilon}(\tau) = A^T(A_\tau^T)^{-1}(w^{(\tau)}-\varepsilon v)$ with multiplicity $\vol_A(\tau)$. This holds for every maximal cell in $\Delta$ for sufficiently small $\varepsilon$. Now, consider the stable intersection with perturbation vector $v$\rev{---since there is a full-dimensional cone of vectors which induce the same triangulation $\Delta$, see \cite[Proposition 5.2.9]{deloera_2010_triangulations}, the vector $v$ is sufficiently general as a perturbation in the stable intersection.}
$$
\row(A)\cap_{st}L_{A,u}=\lim_{\varepsilon\rightarrow 0}(\varepsilon v+\row(A))\cap L_{A,u}.
$$
The intersection on the right contains the points $\hat{q}_{\varepsilon}(\tau)$ with limit equal to $\hat{q}(\tau)$ as $\varepsilon\rightarrow 0$, where $\tau$ runs over maximal cells in $\Delta_v$. This proves that these points are among the tropical critical points. Since the total multiplicity of these points is $\sum_{\tau\in\Delta_v}\vol_A(\tau)=\vol_A(Q_A)$ these are all the critical points.
\end{proof}
\begin{example}[Hirzebruch surface revisited]
We return to Example \ref{ex: Hirzebruch}
one more time. If we take $\Delta$ to be the regular triangulation with maximal simplices $\{123, 234, 245, 256\}$  as in the figure below, we see that 
$w^{(123)} = (0,0,0,0,0,0)$, 
$w^{(234)} = w^{(245)} = (6,6,0,0,0,0)$, and
$w^{(256)} = (6,6,4,4,4,0)$. 
\begin{figure}[h!]
\centering\includegraphics[width=0.13\textwidth]{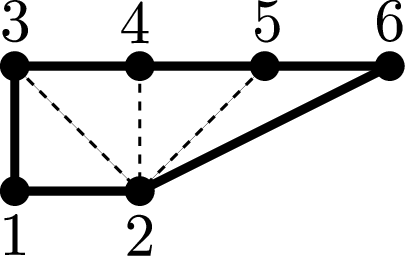}
\end{figure}

The subdivisions
$\Delta_{-w^{(123)}} = \Delta_{-w^{(234)}} = \Delta_{-w^{(245)}}$ have a single
maximal cell $123456$ and $\Delta_{-w^{(256)}}$ has maximal cells $\{12345, 256\}$. In each case, the maximal simplices of $\Delta$ lie in a cell of the corresponding subdivision. Therefore by Theorem \ref{theorem: tropical critical points from triangulation} we recover all tropical critical points:
\begin{align*}
\hat{q}(123) &= A^T(A_{123}^T)^{-1}(0,0,0)^T = (0,0,0,0,0,0)^T = w^{(123)} + 0 e_4 + 0 e_5 + 0 e_6
\\    
\hat{q}(234) &= A^T(A_{234}^T)^{-1}(6,0,0)^T = (6,6,0,0,0,0)^T = w^{(234)} + 0 e_1 + 0 e_5 + 0 e_6
\\
\hat{q}(245) &= A^T(A_{245}^T)^{-1}(6,0,0)^T = (6,6,0,0,0,0)^T = w^{(245)} + 0 e_1 + 0 e_3 + 0 e_6
\\
\hat{q}(256) &= A^T(A_{256}^T)^{-1}(6,4,0)^T = (10,6,12,8,4,0) = w^{(256)} + 4e_1 + 8 e_3 + 4 e_4.
\end{align*}
Since all maximal simplices in $\Delta$ have normalized volume one, this gives tropical critical points with multiplicities $1,2$, and $1$, respectively.
\end{example}
\begin{example}\rev{
Theorem \ref{theorem: tropical critical points from triangulation} does not apply for the matrix $A$ and tropical data vector $w$ in Example \ref{example: square with central point}, for any choice of regular triangulation $\Delta$ of $Q_A$. The unique tropical intersection point is $\hat{q}=0$ with multiplicity $\vol_A (Q_A)=8$, and this point does not lie on any of the cones $C_{\tau}$.}
\end{example}
\rev{In Section \ref{section: general position} we will look at instances under which Theorem \ref{theorem: tropical critical points from triangulation} applies and thus completely solves the tropical maximum likelihood problem. In the remainder of this section we consider two practical} applications of Theorem \ref{theorem: tropical critical points from triangulation}. For convenience, we introduce one further definition: for $\mathcal{O}\subseteq [n]$ and $\tau\in M(A)$, we say that $\tau$ \textit{has an $\mathcal{O}$-basis exchange} if for all $j\in\tau\backslash\mathcal{O}$ there exists $i\in\mathcal{O}\backslash \tau$ such that $\tau-j+i\in M(A)$. Note that if $\tau$ has an $\mathcal{O}(w)$-basis exchange, then $w^{(\tau)}=0$, where \rev{we recall that $\mathcal{O}(w) = \{i \in [n] \, : \, w_i = 0\}$}.

\OContainsBasis*
\begin{proof}
Let $\Delta_\omega$ be an arbitrary regular triangulation of $Q_A$ and $\tau\in\Delta_\omega$ a maximal cell. Since $\mathcal{O}(w)$ contains a basis of $M(A)$, $\tau$ has a $\mathcal{O}(w)$-basis exchange and thus $w^{(\tau)}=0$. The subdivision $\Delta_0$ induced by the constant weight $w^{(\tau)}=0$ is trivial, and thus $\tau$ lies in a cell of $\Delta_0$. This holds for every maximal cell of $\Delta_\omega$ and thus Theorem \ref{theorem: tropical critical points from triangulation} applies. But, since $w^{(\tau)}=0$ the tropical critical points are $\hat{q}(\tau)=0$ for all cells.
\end{proof}
We now move to the second application. \rev{Recall that a} subset $\sigma\subseteq[n]$ is called a \textit{face} of $Q_A$ if there exists a linear form $\psi$ on $\R^k$ such that $\psi(a_i)$ is maximized precisely for all $i\in\sigma$. 

\SolutionForUniformMatroids*
\begin{proof}
We start with the case when $\mathcal{O}(w)$ is not a face of $Q_A$. If $\mathcal{O}(w)$ contains a basis of $M(A)$, then we are done by
Theorem \ref{theorem: O contains basis zero solutions}. If $\mathcal{O}(w)$ does not contain
a basis of $M(A)$, then it must be contained in some basis since $M(A)$ is a uniform matroid. 
Consider the subdivision $\Delta_{e_{\mathcal{O}(w)}}$ and a regular triangulation $\Delta_{\omega}$ that refines it. We claim that $\mathcal{O}(w)$ does not lie in a maximal cell of the subdivision $\Delta_{e_{\mathcal{O}(w)}}$. Otherwise, by Theorem \ref{theorem: definition of regular subdivision} there exists a basis $\tau\supseteq \mathcal{O}(w)$ such that
$$
A^T(A_\tau^T)^{-1}(e_{\mathcal{O}(w)})_{\tau} \leq e_{\mathcal{O}(w)}.
$$
But then the linear function $\psi = \langle (A_\tau^T)^{-1}(e_{\mathcal{O}(w)})_\tau,\,\cdot\rangle$ on $\R^k$ is maximized on $(a_i)_{i\in\mathcal{O}(w)}$, in contradiction with $\mathcal{O}(w)$ not being a face. Since $\mathcal{O}(w)$ does not lie in a cell of $\Delta_{e_{\mathcal{O}(w)}}$, it does not lie in a cell of the refinement $\Delta_{\omega}$ either. As a result, no maximal cell $\tau\in\Delta_\omega$ fully contains $\mathcal{O}(w)$ and thus they all have a $\mathcal{O}(w)$-basis exchange. This results in $w^{(\tau)}=0$ for all these maximal cells, and thus $\hat{q}(\tau)=0$ as in the proof of Theorem \ref{theorem: O contains basis zero solutions}.

Now assume $\mathcal{O}(w)$ contains a face of $Q_A$. Every maximal cell $\tau\in\Delta_\omega$ that does not contain $\mathcal{O}(w)$ has a $\mathcal{O}(w)$-basis exchange and thus contributes a tropical critical point $\hat{q}(\tau)=0$ with multiplicity $\vol_A(\tau)$ as above. We now show that for all maximal cells $\tau\in\Delta_\omega$ that contain $\mathcal{O}(w)$, $\tau$ lies in a cell of $\Delta_{-w^{(\tau)}}$ when the condition on the tropical data vector $w$ is satisfied. Applying Theorem \ref{theorem: tropical critical points from triangulation} then completes the proof. We write the condition explicitly for an entry $i\in[n]\setminus\tau$:
\begin{equation}\label{eq: condition for tau}
\sum_{j\in\tau} (A^T(A_\tau^T)^{-1})_{ij} w^{(\tau)}_j \geq w^{(\tau)}_i.
\end{equation}
By definition $w_j^{(\tau)}\geq\min\lbrace w_z:w_z>0\rbrace$ for all $j\in\tau\setminus\mathcal{O}(w)$. Furthermore, since $M(A)$ is uniform we have $w_j^{(\tau)} = \min\lbrace w_j , w_z : z\in[n]\setminus\tau\rbrace\leq W$ for all $j\in\tau$, where $W$ is any value such that the tropical data vector $w$ has $k$ entries $w_z\leq W$. Using these inequalities on $w^{(\tau)}_{\tau}$, we can bound the left hand side in \eqref{eq: condition for tau} from below as
$$
\sum_{j\in\tau}(A^T(A_\tau^T)^{-1})_{ij}w_j^{(\tau)} \geq \min\{w_z>0\}\cdot\Big(\sum_{j\in\mathcal{J}_+} (A^T(A_\tau^T)^{-1})_{ij}\Big) + W\cdot\Big(\sum_{j\in\mathcal{J}_-} (A^T(A_\tau^T)^{-1})_{ij}\Big),
$$
where $\mathcal{J}_{+}=\{j\in \tau \setminus\mathcal{O}(w) : (A^T(A_\tau^T)^{-1})_{ij}>0\}$ and $\mathcal{J}_-=(\tau\setminus\mathcal{O}(w))\setminus\mathcal{J}_{+}$. We remark that if $\mathcal{J}_-$ is nonempty, then there exists a tropical data vector $w$ for which this inequality is tight. Next, since $w_i^{(\tau)}=\max\{w_j^{(\tau)}:j\in\tau\rbrace$ for all $i\in[n]\setminus\tau$, we have $w_i^{(\tau)}\leq W$ as an upper bound for the \rev{right hand side} in equation \eqref{eq: condition for tau}. Combining the bounds for both sides of the inequality, we obtain the implication
\begin{align*}
\text{\eqref{eq: condition for tau}}
&\,\,\Longleftarrow\,\, \min\{w_z>0\}\cdot\Big(\sum_{j\in\mathcal{J}_+} (A^T(A_\tau^T)^{-1})_{ij}\Big) + W\cdot \Big(\sum_{j\in\mathcal{J}_-} (A^T(A_\tau^T)^{-1})_{ij}\Big) \geq W
\\
&\iff\,\, \min\{w_z>0\} \cdot \underbrace{\frac{\sum_{j\in\mathcal{J}_+}(A^T(A^T_{\tau})^{-1})_{ij}}{1-\sum_{j\in\mathcal{J}_-}(A^T(A^T_{\tau})^{-1})_{ij}}}_{c(A,\mathcal{O},\Delta,\tau,i)}\geq W.
\end{align*}
Here $c(A,\mathcal{O},\Delta,\tau,i)$ is some number that depends on $(A,\mathcal{O},\Delta,\tau,i)$. Minimizing over all regular triangulations $\Delta$ that refine $\Delta_{e_{\mathcal{O}(w)}}$, all maximal cells $\tau\supset\mathcal{O}(w)$ in $\Delta$ and all $i\in[n]\setminus\tau$, one obtains a constant $c_{A,\mathcal{O}}$ that only depends on matrix $A$ and the face $\mathcal{O}(w)$. It remains to prove that this constant satisfies $c_{A,\mathcal{O}}\geq 1$; we show that in fact $c(A,\mathcal{O},\Delta,\tau,i)\geq 1$. Since $\mathcal{O}(w)$ is a face of $Q_A$, we know that $\mathcal{O}(w)$ lies in a cell of $\Delta_{e_{\mathcal{O}(w)}}=\Delta_{-e_{[n]\setminus\mathcal{O}(w)}}$. Thus for every maximal simplex $\tau\supset\mathcal{O}(w)$ in a regular triangulation $\Delta$ that refines $\Delta_{e_{\mathcal{O}(w)}}$, we have
\begin{align*}
&e_{[n]\setminus\mathcal{O}(w)} \leq A^T(A^T_{\tau})^{-1}(e_{[n]\setminus\mathcal{O}(w)})_{\tau} 
\\
\iff &1\leq \sum_{j\in\mathcal{J}_+} (A^T(A_\tau^T)^{-1})_{ij} + \sum_{j\in\mathcal{J}_-} (A^T(A^T_{\tau})^{-1})_{ij}  &\text{~for all $i\in[n]\setminus\tau$}
\\
\iff & 1 \leq \frac{\sum_{j\in\mathcal{J}_+} (A^T(A_\tau^T)^{-1})_{ij}}{1-\sum_{j\in\mathcal{J}_-} (A^T(A_\tau^T)^{-1})_{ij}}=c(A,\mathcal{O},\Delta,\tau,i) &\text{~for all $i\in[n]\setminus\tau$}.
\end{align*}
This implies in particular that $c_{A,\mathcal{O}}\geq 1$ and completes the proof.
\end{proof}
We remark that if $\mathcal{O}(w)$ is a codimension-one face, then the $c_{A,\mathcal{O}}$-condition on $w$ is always satisfied.

\section{Further Examples}
\label{section: further examples}
In this section we give some further examples to illustrate the results of the previous section. We start with the case where $X_A$ is a curve. 
\begin{corollary}[Monomial curves] \label{corollary: curves}
Let 
$$
A = \begin{pmatrix} 1 & 1 & \cdots & 1 \\
a_1 & a_2 & \cdots & a_n    
\end{pmatrix}
$$
where $a_1 < a_2 < \cdots < a_n$ are integers,
and let $w \geq 0$ be a tropical data vector where $w_{min} = \min_{i \in [n]}\{w_i > 0\}$.
Then the tropical critical points are determined as follows: 
\begin{itemize}
 \item[i)] if $w_1 = 0 < w_i$ for 
 all $2 \leq i \leq n$, then the tropical critical points are
 $\hat{q}(12) = w_{min} \cdot (0,1,\frac{a_3-a_1}{a_2-a_1}, \ldots, \frac{a_n-a_1}{a_2-a_1})$ with multiplicity $a_2-a_1$,
 and $\hat{q}=(0,0,\ldots, 0)$ with multiplicity $a_n-a_2$;
 \item[ii)] if $w_n = 0 < w_i$ for 
 all $1 \leq i \leq n-1$, then the tropical critical points are
 $\hat{q}((n-1)n) = w_{min} \cdot (\frac{a_n-a_1}{a_n-a_{n-1}}, \ldots, \frac{a_n-a_{n-2}}{a_n-a_{n-1}}, 1, 0)$ with multiplicity $a_n-a_{n-1}$,
 and $\hat{q}=(0,0,\ldots, 0)$ with multiplicity $a_{n-1}-a_1$;
 \item[iii)] otherwise, $\hat{q}=(0,0,\ldots,0)$ is the unique tropical critical point with multiplicity $a_n-a_1$.
\end{itemize}
\end{corollary}
\begin{proof}
We start by noting that $\deg(X_A) = \vol_A(Q_A) = a_n-a_1$. Since $M(A)$ is 
uniform of rank two, Theorem \ref{theorem: solution for uniform matroids} applies. \rev{For statement i) we know that $\mathcal{O}(w)=\{1\}$ is a face of $Q_A$ and thus case 2 of Theorem \ref{theorem: solution for uniform matroids} applies.} There are two entries $w_i$ such that $w_i \leq w_{min}$ (where we take $c_{A,\mathcal{O}}=1$) and the regular
triangulation $\Delta$ with maximal cells $\{12, 23, \ldots, (n-1)n\}$ refines the subdivision 
$\Delta_{e_{\mathcal{O}(w)}} = \Delta_{e_1}$ with maximal cells $\{12, 234\cdots n\}$.
The simplices in $\Delta$ give us $\hat{q}(12) = w_{min} \cdot (0,1,\frac{a_3-a_1}{a_2-a_1}, \ldots, \frac{a_n-a_1}{a_2-a_1})$ with multiplicity $a_2-a_1$
 and $(0,0,\ldots, 0)$ with multiplicity $a_n-a_2$. \rev{Statement ii) is proven similarly. Statement iii) follows from the first case of} Theorem \ref{theorem: solution for uniform matroids} when $w_i = 0$ for some $1 < i < n$ or from Theorem \ref{theorem: O contains basis zero solutions} when $w_i=0$ for more than one element $i\in[n]$.
\end{proof}
Next we have a look at certain toric surfaces. 
\begin{corollary}[Convex polygons] 
\label{cor: polygons} 
Let 
$$
A = \begin{pmatrix} 1 & 1 & \cdots & 1 \\
a_1 & a_2 & \cdots & a_n    
\end{pmatrix} \,\,\, \rev{\text{~with~}a_i \in \Z^2}
$$
form the cyclically ordered vertices of a convex polygon, and let $w \geq 0$
be a tropical data vector where $w_{min} = \min_{i \in [n]}\{w_i > 0\}$.
Then, 
\begin{itemize}
 \item[i)] if, without loss of generality, $w_1 = 0 < w_i$ for all $2 \leq i \leq n$, then there exists a constant $c\geq 1$ such that for all 
 tropical data vectors $w$ that have $3$ entries with $w_i \leq c \cdot w_{min}$, the tropical critical points are $\hat{q}(12n)$ with multiplicity $\vol_A(12n)$ and $\hat{q}=(0,0,\ldots,0)$ with multiplicity $\vol_A(23\cdots n)$;
 \item[ii)] if, without loss of generality, $w_1=w_2=0 < w_i$ for all $3 \leq i \leq n$, then the tropical critical points are $\hat{q}(12n)$
 with multiplicity $\vol_A(12n)$, and $\hat{q}=(0,0,\ldots, 0)$
with multiplicity $\vol_A(23\cdots n)$;
\item[iii)] if $\{i \, : \, w_i = 0\}$ 
is not a vertex or edge of $Q_A$, then
$(0,0,\ldots, 0)$ is the unique tropical critical point with multiplicity $\vol_A(Q_A)$.
\end{itemize}
\end{corollary}
\begin{proof}
Because $M(A)$ is a uniform matroid of rank three Theorem \ref{theorem: solution for uniform matroids} applies. In the first case, $\Delta_{e_{\mathcal{O}(w)}} = \Delta_{e_1} = \{12n, 23 \cdots n\}$ and the result follows. In the second case, $\Delta_{e_{\mathcal{O}(w)}} = \Delta_{e_{12}} = \{123n, 34 \cdots n\}$, and we can take a regular triangulation refining $\Delta_{e_{12}}$
which contains the simplices $12n$ and $23n$, and the result follows. Again, the third case is a direct application 
of Theorem \ref{theorem: solution for uniform matroids}.
\end{proof}
\begin{example}
We illustrate Corollary \ref{cor: polygons} with an example. 
Consider the pentagon with the tropical data vector $w$
$$
A = \begin{pmatrix}
1&1&1&1&1\\
0&1&2&1&0\\
0&0&1&2&1
\end{pmatrix}
\quad\text{~and~}\quad
w=(0, w_2, w_3, w_4, w_5)
\quad\quad\raisebox{-0.38\height}{\includegraphics[width=0.1\textwidth]{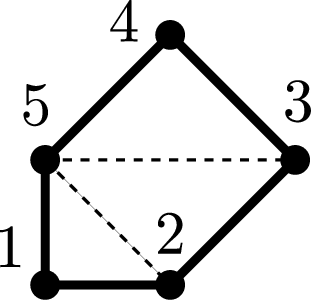}}
$$
The regular subdivision $\Delta_{e_{\mathcal{O}(w)}} = \Delta_{e_1}$ consists of the two maximal cells $125$ and $2345$. The 
triangulation in the above figure refines this subdivision. We compute 
$$
A^T(A_{125}^{-1})^T \, = \, 
\begin{pmatrix}
    1 & 0 & 0\\
    0 & 1 & 0\\
    -2 & 2 & 1 \\
    -2 & 1 & 2 \\
    0 & 0 & 1
\end{pmatrix},
$$
and observe that $\mathcal{J}_+ = \{2,5\}$
and $\mathcal{J}_- = \emptyset$ for both $i=3$ and $i=4$. The third and fourth rows of the above matrix gives $c_{A, \mathcal{O}} = \min(1+2, 2+1) = 3$. An example of a tropical data vector 
which satisfies the condition in Theorem \ref{theorem: solution for uniform matroids} is $w=(0,4,10,6,5)$.
Note that $w_1, w_2, w_5 \leq 3 \cdot 4$. Now we can compute that $w^{(125)} = (0,4,5,5,5)$ and $\hat{q}(125) = (0,4,13,14,5)$ with multiplicity one. The other tropical critical point is $(0,0,0,0,0)$ with multiplicity four. 

Next we specialize the tropical data vector to $w=(0,0,w_3,w_4,w_5)$. Now
$\mathcal{O}(w) = \{1,2\}$ corresponds 
to a face of $Q_A$ of codimension one. The subdivision 
$\Delta_{e_{12}}$ consists of the two maximal faces $1235$ and $345$. The triangulation in the above figure refines this subdivision as well. 
We see that $w_{125}^{(125)} = (0,0,w_{min})$ where $w_{min} = \min(w_3,w_4,w_5)$ and $w^{(125)} = w_{\min}\cdot (0,0,1,1,1)$. The corresponding
tropical critical points is 
$\hat{q}(125)=w_{\min}\cdot (0,0,1,2,1)$ and has multiplicity one. The remaining tropical critical point is $\hat{q}=(0,0,0,0,0)$ with multiplicity four. 
\end{example}
We finish with a three-dimensional example. This corresponds to the independence model of a binary and a ternary random variable.
\begin{example}
  Consider  
  $$
A = \begin{pmatrix}
1&1&1&1&1&1\\
0&1&0&0&1&0\\
0&0&1&0&0&1 \\
0&0&0&1&1&1
\end{pmatrix}
\quad\text{~and~}\quad
w=(w_{11}, w_{12}, w_{13}, w_{21}, w_{22}, w_{23})
\quad\quad \raisebox{-0.35\height}{\includegraphics[width=0.18\textwidth]{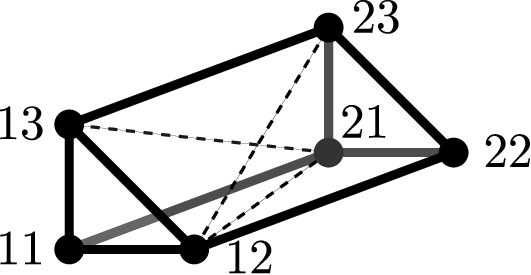}}
$$
The polytope $Q_A$ is the triangular prism shown in the figure, and has $\vol(Q_A) = 3$. We will take $w =(0,1,1,1,2,4)$ and consider the regular triangulation $\Delta$ consisting 
of maximal simplices $\tau_1=\{13, 21, 22, 23\}$, $\tau_2=\{12, 13,21, 22\}$, and $\tau_3=\{11,12,13,21\}$. We compute the corresponding vertices 
\begin{align*}
w^{(\tau_1)} = (0,1,0,0,1,0), \qquad w^{(\tau_2)} = (0,0,1,0,0,1), \qquad
w^{(\tau_3)} = (0,1,1,1,1,1). 
\end{align*}
For the first two simplices $\Delta_{-w^{(\tau)}}$ is the trivial subdivision
and for the last one the corresponding
subdivision has maximal cells $\{11,12,13,21\}$
and $\{12,13,21,22,23\}$. We see that 
each maximal simplex $\tau$ of $\Delta$ 
is lies in a cell of $\Delta_{-w^{(\tau)}}$. Using Theorem \ref{theorem: tropical critical points from triangulation} we can list all the tropical critical points:
\begin{align*}
\hat{q}(\tau_1) = (0,1,0,0,1,0), \qquad \hat{q}(\tau_2) = (0,0,1,0,0,1), \qquad \hat{q}(\tau_3) = (0,1,1,1,2,2).
\end{align*}
Since $A$ is a totally unimodular matrix, 
each of the above critical points has multiplicity one. 
\end{example}

\section{Tropical iterative proportional scaling}
\label{section: tIPS}
In this section, we report some observations on the tropicalization of an algorithm that computes the classical MLE for toric models. Recall that Birch's theorem (Theorem \ref{theorem: Birch}) guarantees that there is a unique positive real point among the critical points of $\ell_u$ on a log-linear model. This is the maximum likelihood estimate (MLE) and can be computed using the \textit{iterative proportional scaling algorithm} (IPS). Following the treatment in \cite{Lectures_on_Alg_Stats}, we can describe the IPS as follows: for a toric model defined by $A\in\mathbb{N}^{k\times n}$ with constant column sums $\alpha$ and with normalized data $u\in\mathbb{R}^n_{\geq 0}$, the algorithm computes a sequence of estimates $\rev{\hat{p}^0,\hat{p}^1},\dots$ based on the recursion:
$$
\hat{p}_j^{t+1} = \hat{p}_j^t \times \prod_{i=1}^k\left(\frac{r_i}{\hat{r}^t_i}\right)^{a_{ij}/\alpha}\quad\text{~where~}\quad\text{~$\hat{r}^t = A\hat{p}^t$ and $r=Au$},
$$
for all $j\in[n]$. Darroch and Ratcliff \cite{IPS72} proved that this algorithm converges to the MLE for any choice of initial distribution $\rev{\hat{p}^0}\in X_A \cap \Delta_{n-1}$. Note that a matrix $A'\in\mathbb{Z}^{k\times n}$ with $(1,\dots,1)$ in its row span can always be transformed to a nonnegative matrix $A$ with constant column sums such that $\row(A)=\row(A')$. In general, the convergence speed of the algorithm depends on this choice of $A$.

Here we investigate the \emph{tropical iterative proportional scaling algorithm} (tIPS). For this, we tropicalize IPS by replacing standard arithmetic by tropical arithmetic: multiplication becomes summation and addition becomes minimization. For a model defined by $A\in\mathbb{N}^{k\times n}$ with constant column sum $\alpha$ and with tropical data vector $w$, the tIPS computes a sequence of vectors $\hat{q}^{0},\hat{q}^{1},\ldots$ in $\R^n$ based on the recursion: 
$$
\hat{q}^{t+1} = \hat{q}^t + \alpha^{-1}A^T(r-\hat{r}^t) \quad\text{~where~}\quad\begin{cases}
\text{$\hat{r}^t_i = \min\lbrace \hat{q}^t_j : j\in[n]\text{~s.t.~}a_{ij}\neq 0\rbrace$}\\
\text{$r_i = \min\lbrace w_j : j\in[n]\text{~s.t.~}a_{ij}\neq 0\rbrace$}.
\end{cases}
$$
The algorithm terminates at step $t$ if $\hat{r}^t=r$. We note the following relation of tIPS to the tropical maximum likelihood estimation problem.
\begin{proposition}
The tropical iterative proportional scaling algorithm satisfies the following properties:
\begin{itemize}
    \item[i)] If $\hat{q}^0\in\row(A)$ then $\hat{q}^t\in\row(A)$ for all $t\geq 0$.
    \item[ii)] The algorithm terminates if it reaches a tropical critical point.
\end{itemize}
\end{proposition}
\begin{proof}
\emph{i)} This follows immediately from the definition of the tIPS algorithm.

\emph{ii)} Let $\hat{p}\in cX_{A}\cap Y_{A,u}$ be a critical point. If $\hat{q}^t=\val(\hat{p})$ is the corresponding tropical critical point, then taking the valuation on both sides of $A\hat{p}=Au$ yields $\hat{r}^t=r$ as required.
\end{proof}
The following two examples show that tIPS converges to a tropical critical point in some but not all cases, and that convergence may depend on the specific matrix $A$ used to parametrize $\row(A')$.

\begin{example} We consider the independent binary random variable model from Example \ref{example: binary random variables - classical solution} and reparametrize matrix $A$ to have nonnegative entries and constant column sums $\alpha=2$:
$$
A=\begin{pmatrix}
2&1&1&0\\0&1&0&1\\0&0&1&1  
\end{pmatrix}\quad\text{~and~}\quad w=(0,2,1,4).
$$
Since $\row(A)$ is not changed by this reparametrization, the tropical critical points are the same as in Example \ref{example: binary random variables - classical solution}, and equal to $\hat{q}_1=(0,0,0,0)$ and $\hat{q}_2=(0,2,3,1)$. We now consider tIPS with initial vector $\hat{q}^0=(0,0,0,0)$. This gives $\hat{r}^0=(0,0,0)$ and we compute $r=(0,2,1)$ from the tropical data vector $w$. We obtain the sequence
$$
\hat{q}^0=\left(\begin{smallmatrix}   0\\0\\0\\0\end{smallmatrix}\right),\quad \hat{q}^1=\tfrac{1}{2} \left(\begin{smallmatrix}   0\\2\\1\\3\end{smallmatrix}\right),\quad \hat{q}^2=\tfrac{3}{4}\left(\begin{smallmatrix}   0\\2\\1\\3\end{smallmatrix}\right),\quad \dots\quad:\quad 
\hat{q}^t=\tfrac{2^t-1}{2^t} \left(\begin{smallmatrix}   0\\2\\1\\3\end{smallmatrix}\right).
$$
The tIPS algorithm converges to the tropical critical point $\hat{q}_2=(0,2,1,3)$.
\end{example}
\begin{example} We consider two parametrizations of the Hirzebruch surface in Example \ref{ex: Hirzebruch} as nonnegative matrices with constant column sum:
$$
A_1 = \begin{pmatrix}
4&3&3&2&1&0\\0&1&0&1&2&3\\0&0&1&1&1&1
\end{pmatrix} \quad\text{~and~}\quad A_2=\begin{pmatrix}
1&1&0&0&0&0\\1&0&3&2&1&0\\1&2&0&1&2&3
\end{pmatrix}.
$$
For $w=(6,8,7,6,4,0)$ and with initial vector $\hat{q}^0=(0,0,0,0,0,0)$, the tIPS converges to $\hat{q}_1=(16,12,12,8,4,0)$ for parametrization $A_1$ and to $\hat{q}_2=(10,6,12,8,4,0)$ for parametrization $A_2$. Here, $\hat{q}_2$ is a tropical critical point while $\hat{q}_1$ is not. 
\end{example}
\begin{question}
For which choices of $A,w,\hat{q}^0$ does tIPS converge to a tropical critical point?
\end{question}

\section{Configurations in general position} \label{section: general position}

At the start of this article, we formulated the general tropical toric maximum likelihood problem as finding the intersection points of a tropical affine space $L_{A,u}$ and a classical linear subspace $\row(A)$. We then narrowed the focus to describing certain `nice' cones $C_{\tau}\subset L_{A,u}$ and their intersections with $\row(A)$, which led to a full solution of the problem for certain instances of $A$ and $u$ in Theorem \ref{theorem: tropical critical points from triangulation}. 

In this section, we look in more detail at the conditions under which this solution strategy applies. As a first result in this direction, Theorem \ref{theorem: theorem 6 applies for full-dimensional cone} shows that for matrices $A$ with columns in general position, i.e. uniform $M(A)$, Theorem \ref{theorem: tropical critical points from triangulation} applies for a full-dimensional set of tropical data vectors $w$. This implies that for a practically relevant class of models $A$ and a full-dimensional, positive-measure subset of tropical data vectors, Theorem \ref{theorem: tropical critical points from triangulation} provides the solution to the tropical toric maximum likelihood problem. 

\begin{theorem}\label{theorem: theorem 6 applies for full-dimensional cone}
For every $A$, the set of tropical data vectors for which Theorem \ref{theorem: tropical critical points from triangulation} applies is the support of a polyhedral fan. If the columns of $A$ are in general position, then this fan contains an $(n-1)$-dimensional cone.
\end{theorem}
\begin{proof} 
We start by showing that the set of tropical data vectors for which Theorem \ref{theorem: tropical critical points from triangulation} applies is a polyhedral fan. First, note that the set of all tropical data vectors forms an $(n-1)$-dimensional polyhedral fan $\mathcal{B}\subset\R^n_{\geq 0}$ in the nonnegative orthant, with cells $\Theta= \{w\in\R^n\,:\, 0=w_{\theta(1)}\leq w_{\theta(2)}\leq\dots\leq w_{\theta(n)}\}$ given by vectors with a fixed order $\theta\in S_n$. Now fix a matrix $A$. The tropical data vectors that satisfy Theorem \ref{theorem: tropical critical points from triangulation} for $A$ with respect to a given triangulation $\Delta$ can be written as
\begin{equation}\label{eq: cone with conditions from Theorem 6 general}
\big\{w\in\mathcal{B}\,:\,A^T(A^T_{\tau})^{-1}w^{(\tau)}_{\tau}\geq w \text{~for all $\tau\in\Delta$} \big\}\subset\R^n_{\geq 0}.
\end{equation}
On any given cell $\Theta\subset\mathcal{B}$, the $\tau$-operator is linear since it is determined by the maximum or minimum of some entries of $w$, whose order is fixed on $\Theta$. The intersection of \eqref{eq: cone with conditions from Theorem 6 general} with $\Theta$ is thus a finite intersection of halfspaces whose boundary contains zero, which is a polyhedral cone. The set of all tropical data vectors for which Theorem \ref{theorem: tropical critical points from triangulation} applies is the union of these cones, ranging over all cells $\Theta\subset \mathcal{B}$ and all triangulations $\Delta$ of $Q_A$. This is the support of a polyhedral fan in $\R^n_{\geq 0}$, where the fan structure can be obtained as the common refinement of all the cones.

The dimension of any cone contained in the fan described above can be at most $(n-1)$, which is the dimension of $\mathcal{B}$. We now show that this maximum dimension is achieved when the columns of $A$ are in general position, in other words, when $M(A)$ is uniform. Let $\theta\in S_n$ be an order such that $\{\theta(i)\}_{i\leq r}$ is not a face of $Q_A$ for every $1<r\leq n$; such an order exists except if $n=k$, but in this case $Q_A$ is a simplex and Theorem \ref{theorem: tropical critical points from triangulation} applies for any tropical data vector. Now let $0=w_{\theta(1)}=\dots=w_{\theta(r)}<w_{\theta(r+1)}\leq\dots\leq w_{\theta(n)}$ and denote $\mathcal{O}(w)=\{\theta(i)\}_{i\leq r}$. If $1<r\leq n$, then since $\mathcal{O}(w)$ is not a face, by the first case of Theorem \ref{theorem: solution for uniform matroids} we know that Theorem \ref{theorem: tropical critical points from triangulation} applies for $w$ with respect to any regular triangulation of $Q_A$. It thus remains to treat the $r=1$ case, for which $\mathcal{O}(w)= \{\theta(1)\}$ is a face of $Q_A$. Here, we will follow a similar approach to the proof of the second case of Theorem \ref{theorem: solution for uniform matroids}: we show that there exists a constant $c>1$ such that if $w_{\theta(k)}\leq c\cdot w_{\theta(2)}$ then Theorem \ref{theorem: tropical critical points from triangulation} applies for $w$.

Let $\Delta$ be any regular triangulation of $Q_A$ that refines the subdivision $\Delta_{e_{\theta(1)}}$. The latter is called a pulling triangulation and by \cite[Lemma 4.3.10]{deloera_2010_triangulations} the cells $\delta\in\Delta_{e_{\theta(1)}}$ for which $\theta(1) \in \delta$ are of the form $\delta = \{\theta(1)\}\cup F$ where $F\subseteq[n]\setminus\{\theta(1)\}$ is a face of 
$\mathrm{conv}(a_i : i \in [n]\setminus \{\theta(1)\})$ that is `visible' from $\theta(1)$.  In other words, there exists an affine function $\psi$ on $\R^k$ which is negative on $a_{\theta(1)}$, zero on $\{a_i\}_{i\in F}$, and positive on all other points. Since the points $\{a_i\}_{i\in[n]\setminus\{\theta(1)\}}$ are in general position, the facets of their convex hull must be simplices and thus all maximal cells $\delta\in\Delta_{e_{\theta(1)}}$ that contain $\theta(1)$ are simplices.

We now check the conditions on $w$ such that every $\tau\in\Delta$ lies in a cell of $\Delta_{-w^{(\tau)}}$. If $\theta(1) \not \in \tau$ then $\tau$ has a basis exchange with $\theta(1)$, by uniformity, and thus $w^{(\tau)}=0$. This means that $\tau$ lies in a cell in the trivial subdivision $\Delta_{0}$. If $\theta(1) \in \tau$ then, as explained above, we know that $\tau$ is also a maximal simplex in $\Delta_{e_{\theta(1)}}=\Delta_{-e_{[n]\setminus\{\theta(1)\}}}$. For every $i\in\sigma = [n] \setminus \tau$, we thus have the strict inequality
$$
A_i^T(A^T_{\tau})^{-1}e_{\tau\setminus\{\theta(1)\}} > 1.
$$
Introducing the index sets $\mathcal{J}_+,\mathcal{J}_-$ as in the proof of Theorem \ref{theorem: solution for uniform matroids}, this can be rewritten as
$$
\sum_{j\in\mathcal{J}_+}(A^T(A^T_{\tau})^{-1})_{ij} + \sum_{j\in\mathcal{J}_-}(A^T(A^T_{\tau})^{-1})_{ij} > 1 \Longleftrightarrow c_{\tau,i}:= \frac{\sum_{j\in\mathcal{J}_+}(A^T(A^T_{\tau})^{-1})_{ij}}{1-\sum_{j\in\mathcal{J}_-}(A^T(A^T_{\tau})^{-1})_{ij}}>1.
$$
The conditions on $w$ such that $\tau$ lies in a cell of $\Delta_{-w^{(\tau)}}$, 
namely, $A_i^T(A^T_{\tau})^{-1}w^{(\tau)}_{\tau} \geq w^{(\tau)}_i$ for all $i\in\sigma$, can similarly be rewritten as
$$
\sum_{j\in\mathcal{J}_+} (A^T(A^T_{\tau})^{-1})_{ij}w^{(\tau)}_j + \sum_{j\in\mathcal{J}_-}(A^T(A^T_{\tau})^{-1})_{ij}w^{(\tau)}_j \geq w_i^{(\tau)}. 
$$
Since $w_{\theta(2)}\leq w^{(\tau)}_j\leq w_{\theta(k)}$ for all $j\in \tau\setminus\{\theta(1)\}$ by uniformity of $M(A)$, the above inequality is implied by the following stricter inequality
$$
w_{\theta(2)}\sum_{j\in\mathcal{J}_+} (A^T(A^T_{\tau})^{-1})_{ij} + w_{\theta(k)}\sum_{j\in\mathcal{J}_-} (A^T(A^T_{\tau})^{-1})_{ij}\geq w_{\theta(k)}, 
$$
or equivalently,
$ c_{\tau,i}\cdot w_{\theta(2)}\geq w_{\theta(k)}$. It thus suffices that $w_{\theta(k)}\leq c\cdot w_{\theta(2)}$ where $1<c\leq c_{\tau,i}$, to guarantee that $\tau$ lies in a cell of $\Delta_{-w^{(\tau)}}$ and that Theorem \ref{theorem: tropical critical points from triangulation} applies for $w$ with respect to $\Delta$. As a result, we have shown that for our choice of $\theta\in S_n$, there exists a constant $c>1$ such that Theorem \ref{theorem: tropical critical points from triangulation} applies for all tropical data vectors that satisfy $0=w_{\theta(1)}<w_{\theta(2)}\leq\dots w_{\theta(n)}$ and $w_{\theta(k)}\leq c\cdot w_{\theta(2)}$. This is an open $(n-1)$-dimensional cone inside $\Theta\subset \mathcal{B}$. If the inequality $w_{\theta(1)} \leq  w_{\theta(2)}$ is non-strict, then Theorem \ref{theorem: tropical critical points from triangulation} still applies due to the derivation above for the case where $r>1$, and thus it applies in particular to the full $(n-1)$-dimensional cone obtained by intersecting $\Theta$ with the halfspace given by the inequality $w_{\theta(k)}\leq c\cdot w_{\theta(2)}$. This completes the proof.
\end{proof}
The following example shows that if the columns of $A$ are not in general position, then it is possible that there exists no full-dimensional cone of tropical data vectors for which Theorem \ref{theorem: tropical critical points from triangulation} applies.
\begin{example} We consider the matrix $A$ and point configuration $Q_A$ below, and show that there does not exist any $(n-1)$-dimensional cone of tropical data vectors for which Theorem \ref{theorem: tropical critical points from triangulation} applies.
$$
A = \begin{pmatrix}
1&1&1&1&1&1\\
0&2&0&2&1&1\\
0&0&2&2&1&1\\
0&0&0&0&2&1
\end{pmatrix}\quad\quad\quad\quad\raisebox{-0.38\height}{\includegraphics[width=0.18\textwidth]{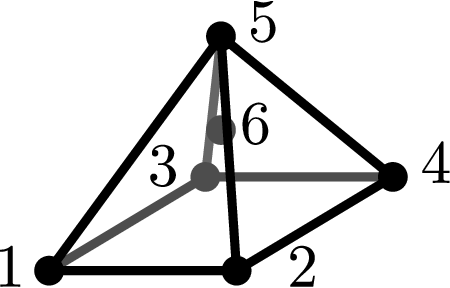}}
$$
Up to symmetry, there are two triangulations of $Q_A$ that need to be checked: $\Delta^{(1)} = \{1245,1345\}$ and $\Delta^{(2)}=\{1256,1356,2456,3456,1236,2346\}$, where the maximal simplices are listed. We confirm that there does not exist an $(n-1)$-dimensional cone by putting one entry of $w$ to zero and all other entries positive and distinct, and then showing that this leads to a failure of the condition that some $\tau\in\Delta$ lies in a cell of $\Delta_{-w^{(\tau)}}$, i.e., we show that $A_i^T(A^T_{\tau})^{-1}w^{(\tau)}_\tau< w^{(\tau)}_i$ for some $\tau\in\Delta$ and $i\in\sigma$.

For $\Delta^{(1)}$ we start with $w_6=0$. In this case we have that $w^{(\tau)}_{1456}=0$ for all $\tau$ and positive otherwise, which leads to a failure of the condition for $\tau=\{1245\}$ and $i=3$ since $A_3^T(A^T_{\tau})^{-1}w^{(\tau)}_\tau<0 <w^{(\tau)}_3$. For all other cases, where $w_6>0$, we note that $6$ lies in the interior of the simplex $145$ from which it follows that $w^{(\tau)}_6=\max(w^{(\tau)}_1,w^{(\tau)}_4,w^{(\tau)}_5)$ for all $\tau$. As a result, for $\tau=\{1245\}$ we find that $A_6^T(A^T_{\tau})^{-1}w^{(\tau)}_\tau<\max(w^{(\tau)}_1,w^{(\tau)}_4,w^{(\tau)}_5)=w^{(\tau)}_6$ whenever either of $w_1,w_2,w_5$ is equal to zero. By symmetry of the configuration, these are all the relevant cases for $\Delta^{(1)}$.

For $\Delta^{(2)}$ we start with $w_6=0$. In this case we find a failure for $\tau=\{1256\}$ and $i=3$ because $A_3^T(A^T_{\tau})^{-1}w^{(\tau)}_\tau<0<w^{(\tau)}_3$, which follows because $6$ lies in the interior of the simplex $235$, and $w^{(\tau)}_6=0$ and $w^{(\tau)}_{25}>0$. For $w_1=0$ and $\tau=\{2456\}$, we find $w^{(\tau)}_{1456}=0$ and $w^{(\tau)}_{23}>0$, which leads to a failure with respect to $i=3$ since $A_3^T(A^T_{\tau})^{-1}w^{(\tau)}_{\tau}<0<w^{(\tau)}_3$. For $w_5=0$ and $\tau=\{2346\}$, we find $w^{(\tau)}_{2356}=0$ and $w^{(\tau)}_{14}>0$, which leads to a failure with respect to $i=1$ since $A_1^T(A^T_{\tau})^{-1}w^{(\tau)}_{\tau}<0<w^{(\tau)}_1$. By symmetry of the configuration, these are all the relevant cases for $\Delta^{(2)}$. 
\end{example}


\bibliographystyle{abbrv}
\bibliography{bibliography.bib}
\bigskip
\bigskip
		
\noindent {\bf Authors' addresses:}
\smallskip

\noindent Emma Boniface, UC Berkeley \hfill \url{eboniface@berkeley.edu}

\noindent Karel Devriendt, MPI-MiS Leipzig \hfill \url{karel.devriendt@maths.ox.ac.uk}

\noindent Serkan Ho\c{s}ten, SFSU San Francisco \hfill 
\url{serkan@sfsu.edu}

\end{document}